\documentclass[reqno,11pt]{amsart}

\usepackage[all]{xy}
\usepackage{graphicx}

\usepackage{amssymb}
\usepackage{amsmath}

\usepackage{mathrsfs}
\usepackage{epsfig}
\usepackage{amscd}
\usepackage{graphicx, color}
\usepackage{xcolor}

\def\E{\ifmmode{\mathbb E}\else{$\mathbb E$}\fi} 
\def\N{\ifmmode{\mathbb N}\else{$\mathbb N$}\fi} 
\def\R{\ifmmode{\mathbb R}\else{$\mathbb R$}\fi} 
\def\Q{\ifmmode{\mathbb Q}\else{$\mathbb Q$}\fi} 
\def\C{\ifmmode{\mathbb C}\else{$\mathbb C$}\fi} 
\def\H{\ifmmode{\mathbb H}\else{$\mathbb H$}\fi} 
\def\Z{\ifmmode{\mathbb Z}\else{$\mathbb Z$}\fi} 
\def\P{\ifmmode{\mathbb P}\else{$\mathbb P$}\fi} 
\def\T{\ifmmode{\mathbb T}\else{$\mathbb T$}\fi} 
\def\SS{\ifmmode{\mathbb S}\else{$\mathbb P$}\fi} 
\def\DD{\ifmmode{\mathbb D}\else{$\mathbb D$}\fi} 

\newcommand{\del}{\partial}
\newcommand{\Cont}{{\operatorname{Cont}}}

\newcommand{\ben}{\begin{enumerate}}
\newcommand{\een}{\end{enumerate}}
\newcommand{\be}{\begin{equation}}
\newcommand{\ee}{\end{equation}}
\newcommand{\bea}{\begin{eqnarray}}
\newcommand{\eea}{\end{eqnarray}}
\newcommand{\beastar}{\begin{eqnarray*}}
\newcommand{\eeastar}{\end{eqnarray*}}
\newcommand{\bc}{\begin{center}}
\newcommand{\ec}{\end{center}}

\theoremstyle{theorem}
\newtheorem{thm}{Theorem}[section]
\newtheorem{cor}[thm]{Corollary}
\newtheorem{lem}[thm]{Lemma}
\newtheorem{prop}[thm]{Proposition}

\theoremstyle{definition}
\newtheorem{defn}[thm]{Definition}
\newtheorem{rem}[thm]{Remark}
\newtheorem{ques}[thm]{Question}

\newtheorem{prob}[thm]{Problem}

\newtheorem*{thm*}{Theorem}

\numberwithin{equation}{section}

\def\R{{\mathbb R}}

\def\E{{\mathbb E}}
\def\Z{{\mathbb Z}}
\def\C{{\mathbb C}}
\def\R{{\mathbb R}}
\def\P{{\mathbb P}}

\def\N{{\mathbb N}}

\def\11{{\mathbb I}}

\def\id{\text{\rm id}}

\def\C{\mathbb{C}}
\def\Z{\mathbb{Z}}

\def\T{\mathbb{T}}

\def\Q{\mathbb{Q}}

\def\E{\ifmmode{\mathbb E}\else{$\mathbb E$}\fi} 
\def\N{\ifmmode{\mathbb N}\else{$\mathbb N$}\fi} 
\def\R{\ifmmode{\mathbb R}\else{$\mathbb R$}\fi} 
\def\Q{\ifmmode{\mathbb Q}\else{$\mathbb Q$}\fi} 
\def\C{\ifmmode{\mathbb C}\else{$\mathbb C$}\fi} 
\def\H{\ifmmode{\mathbb H}\else{$\mathbb H$}\fi} 
\def\Z{\ifmmode{\mathbb Z}\else{$\mathbb Z$}\fi} 
\def\P{\ifmmode{\mathbb P}\else{$\mathbb P$}\fi} 
\def\SS{\ifmmode{\mathbb S}\else{$\mathbb P$}\fi} 
\def\DD{\ifmmode{\mathbb D}\else{$\mathbb D$}\fi} 

\def\R{{\mathbb R}}

\def\E{{\mathbb E}}
\def\Z{{\mathbb Z}}
\def\C{{\mathbb C}}
\def\R{{\mathbb R}}

\def\N{{\mathbb N}}
\def\FF{{\mathcal F}}

\def\CD{{\mathcal D}}

\def\CF{{\mathcal F}}
\def\CG{{\mathcal G}}

\def\CI{{\mathcal I}}
\def\CJ{{\mathcal J}}

\def\CL{{\mathcal L}}
\def\CM{{\mathcal M}}

\def\CO{{\mathcal O}}
\def\CP{{\mathcal P}}

\def\CP{{\mathcal P}}
\def\CS{{\mathcal S}}

\def\CZ{{\mathcal Z}}

\def\EC{\mathfrak{C}}
\def\darr#1{\raise1.5ex\hbox{$\leftrightarrow$}
\mkern-16.5mu #1}

\def\roughly#1{\raise.3ex\hbox{$#1$\kern-.75em
\lower1ex\hbox{$\sim$}}}

\def\opname#1{\mathop{\kern0pt{\rm #1}}\nolimits}

\def\dim{\opname{dim}}
\def\vol{\opname{vol}}

\def\coker{\operatorname{Coker}}

\def\Cont{\operatorname{Cont}}

\def\Diff{\operatorname{Diff}}
\def\Image{\operatorname{Image}}

\def\pr{\text{\rm pr}}

\DeclareFontFamily{U}{MnSymbolC}{}
\DeclareSymbolFont{MnSyC}{U}{MnSymbolC}{m}{n}
\DeclareFontShape{U}{MnSymbolC}{m}{n}{
    <-6>  MnSymbolC5
   <6-7>  MnSymbolC6
   <7-8>  MnSymbolC7
   <8-9>  MnSymbolC8
   <9-10> MnSymbolC9
  <10-12> MnSymbolC10
  <12->   MnSymbolC12}{}
\DeclareMathSymbol{\intprod}{\mathbin}{MnSyC}{'270}

\begin{document}
\quad \vskip1.375truein

\title{Thermodynamic reduction of contact dynamics}

\author{Hyun-Seok Do, Yong-Geun Oh}
\address{Department of Mathematics POSTECH \text{\&}
 Center for Geometry and Physics, Institute for Basic Science(IBS),}
 \email{hyunseokdo@postech.ac.kr}
\address{
Center for Geometry and Physics, Institute for Basic Science (IBS),
77 Cheongam-ro, Nam-gu, Pohang-si, Gyeongsangbuk-do, Korea 790-784
\& POSTECH, Gyeongsangbuk-do, Korea}
\email{yongoh1@postech.ac.kr}

\begin{abstract}A universal algorithm to derive a macroscopic dynamics from the microscopic
dynamical system via the averaging process and symplecto-contact reduction was
introduced by Jin-wook Lim and the author in \cite{lim-oh}. They
apply the algorithm to derive  non-equilibrium thermodynamics
from the statistical mechanics  utilizing the relative information entropy as a
generating function of the associated thermodynamic equilibrium. 
In the present paper, we apply this algorithm
to the contact Hamiltonian dynamical systems. We describe a procedure of obtaining a 
discrete set of dynamical invariants of the given contact Hamiltonian system, 
or more generally of a contact multi-Hamiltonian system in a canonical way by 
deriving a (finite-dimensional non-equilibrium) \emph{thermodynamic system}. 
We call this reduction the \emph{thermodynamic reduction} of contact dynamics. 
\end{abstract}

\thanks{This work is supported by the IBS project \# IBS-R003-D1}


\maketitle

\tableofcontents

\section{Introduction}
A dynamical system is said to be \emph{volume-preserving} or \emph{measure-preserving}
 if the volume of an open subset of the phase space  is invariant under the flow.  
Poincar\'e recurrence theorem is a celebrated theorem that holds for any such
measure-preserving system which reads as follows.

\medskip

\noindent{\bf Poincar\'e Recurrence Theorem:} \emph{
 If a flow preserves volume and has only bounded orbits, then, for each open set, 
 any orbit that intersects this open set intersects it infinitely often.}

\subsection{Contact dynamics is not measure-preserving}

It is well-known that symplectic Hamiltonian flow is measure-preserving with respect to
the canonical Liouville measure. Similarly the Reeb flow of a given contact form $\lambda$
on a contact manifold $(M,\xi)$ of dimension $2n+1$ is also measure-preserving with respect to
the canonical volume form $\lambda \wedge (d\lambda)^n$. Therefore both 
flows on a closed manifold satisfy the  Poincar\'e recurrent theorem.

On the other hand, the general contact flows $\psi_t$ of $(M,\xi)$ are not measure-preserving and hence
Poincar\'e recurrence fails for the general contact flows. While there are many
similarities in their Hamiltonian calculi, this dynamical perspective
makes a stark difference of general contact dynamics from that of either symplectic 
Hamiltonian dynamics or of Reeb dynamics. This being said, a fundamental question 
in contact dynamics is the question on how one can systematically study this dissipative aspect 
of contact dynamics in terms of the underlying \emph{contact Hamiltonian
formalism}.  This paper is the first one in a  series of papers in preparation
towards this goal.

\emph{Throughout the paper, we will assume that $M$ is orientable and an orientation
is fixed unless otherwise mentioned, except when we feel need to be emphasized.}

More precisely, let us consider a given contact manifold $(M,\xi)$ and consider 
the set of contact forms $\lambda$ i.e., those $\lambda$ satisfying $\ker \lambda = \xi$.
We define
\be\label{eq:C(M,xi)}
\mathfrak C(M,\xi) = \{\lambda \in \Omega^1(M) \mid \ker \lambda = \xi\}.
\ee
A self-diffeomorphism $\psi$ of contact manifold $(M,\xi)$ is called a \emph{contactomorphism}
if it preserves $\xi$, i.e., if $d\psi(\xi) \subset \xi$.  We denote by $\text{Cont}(M, \xi)$ the set of contactomorphisms, by $\text{Cont}_+(M, \xi)$ 
 the set of (co)orientation-preserving contactomorphisms, and by $\text{Cont}_0(M, \xi)$ 
 the identity component of $\text{Cont}_+(M,\xi)$. We emphasize that 
this definition depends only on $\xi$ and does not require presence of contact forms.
When we are given a contact form $\lambda \in \mathfrak{C}(M,\xi)$, 
which presumes the contact structure $\xi$ is coorienatable,  $\psi$ is a 
(co)orientation-preserving  contactomorphism if and only if
$$
\psi^*\lambda = f \lambda
$$
for some positive function $f = f_{(\psi;\lambda)}: M \to \R$: This is called the \emph{conformal factor} of $\psi$
relative to the contact form $\lambda$. \emph{We will always choose the coorientation compatible 
with the given orientation} so that for a positive contact form $\lambda$, we have the decomposition
$$
T_pM = \xi_p \oplus \R\langle R_\lambda(p) \rangle
$$
as an oriented vector space at any $p \in M$, where $R_\lambda$ is the Reeb vector field
associated to $\lambda$. (See Definition \ref{defn:positive-form}
for the standard definition of Reeb vector field $R_\lambda$ in contact geometry.)
The second-named author has been calling its logarithm
\be\label{eq:conformal-exponent}
g_{(\psi;\lambda)} : = \log f_{(\psi;\lambda)}
\ee
the \emph{conformal exponent} in his series of papers on contact instantons and others.
(See \cite{oh:contacton-Legendrian-bdy}, \cite{oh-yso:spectral},  in particular
for such practice.) We denote by $\mu_\lambda$ the measure associated to the volume form 
$\lambda \wedge (d\lambda)^n$.

\emph{Contact topology} is the study of contact distribution $\xi$. There are various reasons why one
 should enlarge the contact phase space to 
 the space of  the pairs $(\lambda,\psi)$ of contact forms $\lambda$ and diffeomorphisms 
 $\psi$ as two independent
information systems in the point of view of \emph{contact dynamics}, especially 
towards the thermodynamic formalism of \emph{contact dynamics}. 
(We postpone elaborate discussion on this formalism until the sequel \cite{do-oh:formalism} of the present paper. See Appendix \ref{sec:peek} for a brief peek of \cite{do-oh:formalism}.)

We denote by $\mathfrak{C}(M)$ to be the set of all nondegenerate one-forms, i.e., all
contact forms on $M$ with the associated contact distribution \emph{not specified}.

\begin{defn}[Contact form and its mass] 
Let $\mathfrak{C}(M)$ be the set of contact forms.
We define the mass function $V: \mathfrak{C}(M) \to \R$ by
$$
V(\lambda): = \int_M \lambda \wedge (d\lambda)^n.
$$
\begin{enumerate}
\item 
We denote by $\mathfrak{C}^\pm(M)$ the set of contact forms with positive and negative
masses respectively. When the mass is positive, we also call $V(\lambda)$ the \emph{volume}
of $\lambda$.
\item We say a contact form \emph{normalized} if $V(\lambda) = 1$ and define the subset
\be\label{eq:mass=1}
\mathfrak{C}^1(M): = \{ \lambda \in \mathfrak C^+(M) \mid V(\lambda) = 1 \}.
\ee
\end{enumerate}
\end{defn}
We would like to mention that the total mass $V(\lambda)$ is unchanged under the action by
diffeomorphisms by the change of variables:
\be\label{eq:diff-invariance}
V(\lambda) = V(\varphi^*\lambda)
\ee
for any orientation-preserving diffeomorphism $\varphi:M \to M$, where the orientation is 
given by the contact volume form $\lambda \wedge (d\lambda)^n$. 

\begin{defn}[Contact pair] We call a pair $(\psi,\lambda)$ a \emph{contact pair} if
they satisfy the relation $\psi^*\lambda = f \lambda$ for a nowhere 
vanishing $f$. We call it \emph{positive} if $f > 0$. 
When $f > 0$, we rename the logarithm 
$g_{\psi,\lambda} : = \log f_{\psi,\lambda}$ the \emph{thermodynamic potential}
of the pair $(\psi,\lambda)$.
\end{defn}

We refer readers to \cite{do-oh:formalism} for the justification of our renaming of
the conformal exponent $g_{(\psi,\lambda)}$ as thermodynamic potential
of the pair $(\psi,\lambda)$, or see Appendix \ref{sec:peek}.

\subsection{Contact dynamics and its thermodynamics reduction}

In \cite{lim-oh}, Lim and the present author derived non-equilibrium thermodynamics 
from the statistical mechanics as a symplecto-contact reduction. A fundamental aspect of
statistical mechanics phase space (SPS) used in that derivation is the probability distribution
space. Each local observable system
$$
\CF = \{F_1, \cdots, F_k\}
$$
then induces a system of functions on the probability distribution space $\CP(M)$
$$
\CO_\CF: = \{\langle F_i \rangle_\rho\}_{i=1}^k
$$
by considering the `observation' $\CO_{F_i}:  \CP(M) \to \R$ defined by 
$$
\CO_{F_i}(\rho): = \langle F_i \rangle_\rho = \int F_i \, \rho.
$$
Regarding this as a moment map (this is their first motto ``\emph{Observation is a moment map}'')
and applying the symplectic reduction, the system
can be reduced to the intermediate \emph{mesoscopic phase space} on which 
the \emph{reduced relative information entropy function}, denoted by
$\CS^{\text{\rm red}}_\CF$, generates a thermodynamic equilibrium.
This is the  second motto of \cite{lim-oh} which reads \emph{``Relative information entropy is 
the generating function of  a thermodynamic equilibrium.''}. 
We call the function $\CS^{\text{\rm red}}_\CF$
the \emph{thermodynamic generating function} of
the multi-Hamiltonian system $\CF$. (See \cite{lim-oh} or Section \ref{sec:reduced-entropy}
for the explanation.)

In this derivation of \cite{lim-oh}, there is
no room to accommodate one of the fundamental thermodynamic variable,  the volume $V$,
 \emph{in the microscopic level} of the given particle system. 
Only after another information system of 
the state of \emph{expandable container} is introduced, one can involve the volume
in the thermodynamics.  In the physical reality, it is natural to consider the gas confined
in an expandable container and study the change of its volume depending on the
various thermodynamic quantities.  

We will apply the above construction to an observable system defined on the contact manifolds
in the point of view of contact dynamics (on compact contact manifolds without boundary).
Then all these aspects already naturally
come into play in the canonical fashion by regarding the choice of contact forms
as another \emph{ information system} like the state of expandable containers in the ideal gas:
First, the contact form $\lambda$ associates a natural measure $\mu_\lambda$ 
induced by the top-degree form $\lambda \wedge (d\lambda)^n$ and all contactomorphisms
preserves the total mass: For any relatively compact open subset $U \subset M$, we define
\be\label{eq:measure-mulambda}
\mu_\lambda(U): = \int_U \lambda\wedge (d\lambda)^n
\ee
which canonically and uniquely extends to a measure. Secondly, we highlight the fact that
the choice of the local observable system $\{F_1, \ldots, F_N\}$ is a datum independent of
the choice of contact forms. In this regard, the pair $(\lambda, \{F_1, \dots, F_N\})$
should be regarded as the composition of two information system. We will call
the pair $(\lambda, \CF)$ a (random) \emph{contact multi-Hamiltonian system} on $M$.

\subsection{Where are we heading for?}

It was already indicated in \cite[Remark 2.2]{lim-oh} that the construction 
given therein should be applied to any smooth manifold with a reference measure 
given or in the relative context. In ibid., it has been used to derive the non-equilibrium thermodynamics
from the statistical mechanics from the first principle in a canonical way.

The main purpose of the present paper is reversed, i.e., is to go the other way around: 
we would like to utilize the aforementioned general scheme to systematically 
involve the relative information theory or the thermodynamic formalism into the 
study of contact dynamics.
Here the Kullback-Leibler divergence 
 $$
 D_{\text{\rm KL}}(\mu_{\lambda'}, \mu_\lambda)
 $$
 of the pair of measures $(\mu_{\lambda'}, \mu_\lambda)$ will
play a fundamental role in  that it
appears as the $(n+1)$-th power of \emph{conformal factor} of $\lambda'$
relative to $\lambda$ \emph{when $\lambda$ and $\lambda'$ have
the same contact structure.}  This is the starting point of our search for 
a thermodynamic formalism adapted to the contact dynamics, especially towards a
systematic study of \emph{dissipative aspect} of general contact dynamics.
The present paper is the first one towards this goal. We refer readers to
a sequel \cite{do-oh:formalism} in preparation for more details on this formalism.
For readers' convenience, we provide a sneak preview of \cite{do-oh:formalism} in Appendix 
\ref{sec:peek}.

\bigskip

\noindent{\bf Acknowledgement:}
This paper is essentially a survey paper for our ongoing joint research 
on the thermodynamic formalism \cite{do-oh:formalism} in preparation. It provides 
the general framework to and perspective on our research in \cite{do-oh:formalism}.


\section{Contact kinetic theory phase space: big phase space}

Let $M$  be a compact connected smooth manifold of dimension $2n+1$ with \emph{nonempty}
set of contact forms
$$
\EC (M)=\{\lambda \in \Omega^1(M) \mid \lambda \wedge (d\lambda)^n \text{ is nonvanishing }\}.
$$
It follows that $\EC(M)$ is an open subset of $\Omega^1(M)$ and so its tangent space 
is canonically identified with $\Omega^1(M)$ itself.  Recall that the set $\EC(M)$ is empty
\emph{unless $M$ carries a contact form}.

 When $M$ is equipped with an orientation, we also define 
$$
\mathfrak C^+(M)=\{\lambda \in  \mathfrak C(M) \mid  \lambda \wedge (d\lambda)^n  \text{
\rm is positive}\} \subset \Omega^1(M).
$$
We denote by this inclusion map by $\Upsilon^{\text{\rm big}}$ putting
\be\label{eq:big-Upsilon}
\Upsilon^{\text{\rm big}}(\lambda)  = \lambda
\ee
regarding $\mathfrak C^+(M)$ as an open subset of $\Omega^1(M)$.
We also consider the subset $\mathfrak{C}_1(M)$
consisting of normalized contact forms.
We define a measure induced by integrating the volume form $\mu_\lambda = \lambda \wedge 
(d\lambda)^n$,
and putting
$$
\mu_\lambda(U) : = \int_U \lambda \wedge  (d\lambda)^n
$$
which uniquely extends to the measure $\mu_\lambda$. One can check

\begin{lem}\label{lem:diff-actionon-mu} Assume that $M$ is oriented. For any pair 
$$
(\varphi, \lambda) \in  \Diff^+(M) \times  \mathfrak{C}^+(M)
$$
of orientation-preserving diffeomorphism and positive contact form, we have
\be\label{eq:diff-actionon-mu}
\varphi_*\mu_\lambda = \mu_{\varphi*\lambda}.
\ee
\end{lem}
\begin{proof} For each relative compact open subset $U$, we evaluate
\beastar
\varphi_*\mu_\lambda(U) & =& \mu_\lambda(\varphi^{-1}(U))
=  \int_{\varphi^{-1}(U)} \lambda\wedge (d\lambda)^n \\
&= & \int_U \varphi_*(\lambda \wedge (d\lambda)^n) 
=    \int_U \varphi_*\lambda \wedge (d (\varphi_*\lambda))^n \\
& = & \int_U \mu_{\varphi_*\lambda} = \mu_{\varphi_*\lambda}(U).
\eeastar
\end{proof}
Based on this lemma, we will safely abuse the notation $\mu_\lambda$  both for 
a differential form and for the associated measure which should not confuse the readers.

We hope to apply the theory of symplectic reduction to an infinite 
dimensional symplectic manifold, the cotangent bundle of $\mathfrak{C}^+(M)$, 
$$
T^*\mathfrak{C}^+(M)
$$
as the Hamiltonian phase space \cite{lim-oh}. We denote by $(\lambda,\beta^\flat)$ 
an element thereof where $\beta^\flat$ indicates an element of $(\Omega^1(M))^*$
which is the set of  one-currents acting on $\Omega^1(M)$. 
Recall that we have a canonical embedding $\mathfrak X(M) \hookrightarrow (\Omega^1(M))^*$
given by the continuous linear funcitional 
$$
X \mapsto \left(\alpha \mapsto \int_M \alpha(X)\,d\mu_{\lambda_0}\right).
$$
(See \cite{federer}.)
Note that we have \emph{canonical} identification of its tangent space given by
$$
T_\lambda \mathfrak{C}^+(M) \cong \Omega^1(M)
$$ 
induced by the map $\alpha \mapsto \lambda + \alpha$, noting that nondegeneracy of
one-form is an open condition in $C^1$ topology which is obvious 
on compact $M$  and after taking is a suitable $C^\infty$ topology on noncompact $M$.
 So, $\mathfrak{C}^+(M)$ can be considered as an (infinite dimensional) smooth manifold
 modeled with $\Omega^1(M)$ equipped with a suitable Whitney $C^\infty$ topology
  (see \cite[Section 2.4]{Hirsch1976}), for example.

On $T^*\mathfrak{C}^+(M)$, we have the canonical Liouville 1-form $\theta$ defined by
$$
{\theta}_{(\lambda,\beta^\flat)}(v)=\beta^\flat(d\pi_{(\lambda,\beta^\flat)}(v)),
$$
where $(\lambda,\beta^\flat)\in T^*\mathfrak{C}^+(M)$, 
$v\in T_{(\lambda,\beta^\flat)}(T^*\mathfrak{C}^+(M))$, and  the canonical projection
$$
\pi: T^*\mathfrak{C}^+(M) \rightarrow \mathfrak{C}^+(M).
$$
Then $\omega_0: = - d\theta$ defines  a \textit{weak symplectic structure} on 
$T^*\mathfrak{C}^+(M)$.

\begin{defn} (Big contact kinetic theory phase space(b-CKTPS))\\
We call the cotangent bundle  $T^*\mathfrak{C}^+(M)$ equipped with its canonical symplectic structure $\omega_0$ the $\textit{big contact kinetic theory phase space}$ (b-CKTPS).    
\end{defn}

\section{Contact kinetic theory phase space: small phase space}

 Recall each contact form $\lambda$ uniquely defines the Reeb vector field $R_\lambda$
 determined by the equation
 \be\label{eq:defining-Rlambda}
 R_\lambda \intprod d\lambda = 0, \quad R_\lambda \intprod \lambda = 1.
 \ee
\begin{defn}\label{defn:positive-form} Let $(M,\xi)$ be contact manifold with coorientation. We say a contact form $\lambda$
 \emph{positive} (resp. \emph{negative}) if $R_\lambda$ is compatible with the given coorientation.
 \end{defn}
 Let $(M,\xi)$ be a coorientable compact connected contact manifold of dimension $2n+1$. 
 We denote by 
 $$
 \mathfrak C(M,\xi)
 $$
 the set of contact forms $\lambda$ of $\xi$, i.e. of those satisfying ker $\lambda=\xi$ 
 and $\lambda \wedge (d\lambda)^{n}$ is a volume form. By definition, if we fix an auxiliary 
 orientation of $M$, each $\lambda$ induces a canonical (smooth) measure induced by the volume 
 form $\lambda \wedge (d\lambda)^n$, which is independent of the choice of orientations of $M$.
We denote by
\be\label{eq:iota-lambda}
\iota_\xi: \mathfrak{C}(M,\xi) \hookrightarrow \mathfrak{C}(M)
\ee
the canonical inclusion for given contact structure $\xi$ on $M$.

\begin{lem}\label{lem:tangent-space} Assume that $(M,\xi)$ is cooriented.
The subset $\mathfrak{C}(M,\xi)$ 
forms a smooth submanifold of $\mathfrak{C}(M)$ whose tangent space at $\lambda$ is give by
\be\label{eq:Tlambda-CMxi}
T_\lambda \mathfrak{C}(M,\xi) = \{\alpha \in \Omega^1(M) \mid \alpha = h \lambda, \, h \in C^\infty(M,\R)\}.
\ee
\end{lem}

We also denote the set of positive contact forms by $\EC^+(M, \xi)$. Then we can write
any positive contact form $\lambda \in  \mathfrak C^+(M,\xi)$
as $\lambda = f\, \lambda_0$ for a unique $f \in C^{\infty}(M, \mathbb{R}_+)$ and vice versa
with $f^{n+1} = f_{\lambda ; \lambda_0}:=\frac{d \mu_\lambda}{d \mu_{\lambda_0}}$. 
We also define the subset
$$
\mathfrak C^1(M, \xi):=\{\lambda \in \mathfrak C^+(M, \xi) \mid V(\lambda)=1 \}
$$
consisting of normalized contact forms with its volume 1, i.e., $\int_M \mu_\lambda = 1$.

On $T^*\mathfrak{C}^+(M,\xi)$, we have the canonical 1-form $\theta$ defined by
$$
{\theta}_{(x,\beta)}(v)=\beta(d\pi_{(x,\beta)}(v)),
$$
where $(x,\beta)\in T^*\mathfrak{C}^+(M,\xi)$, $v\in T_{(x,\beta)}(T^*\mathfrak{C}^+(M,\xi))$, and  the canonical projection
$$
\pi: T^*\mathfrak{C}^+(M,\xi) \rightarrow \mathfrak{C}^+(M,\xi)
$$
with $\pi(x,\beta)=x$.
Then $\omega_0: = - d\theta$ defines  a \textit{weak symplectic structure} on 
$T^*\mathfrak{C}^+(M,\xi)$ (\cite[Section 48.2]{kriegl-michor} for the precise meaning thereof.)

\begin{defn} (Small contact kinetic theory phase space (s-CKTPS))\\
We call the cotangent bundle  $T^*\mathfrak{C}^+(M,\xi)$ equipped with its canonical 
symplectic structure $\omega_0$ the $\textit{small contact kinetic theory phase space}$ (s-CKTPS).    
\end{defn}
Now we denote by 
$$
\Upsilon^{\text{\rm sm}}_\xi: \mathfrak{C}(M,\xi) \hookrightarrow \Omega^1(M)
$$
the canonical inclusion map. Note that the map depends only on the choice of contact structure $\xi$.

Let us fix a reference contact form $\lambda_0 \in \mathfrak C(M,\xi)$ compatible with given
coorientation of $\xi$.  We recall readers that
the subset $\mathfrak C(M,\xi) \subset \Omega^1(M)$ is neither a linear subspace nor
an open subset of $\Omega^1(M)$.

This give rise to a  (non-canonical) diffeomorphism 
\be\label{eq:small-Upsilon-lambda0}
\CI_{\lambda_0} : C^\infty(M,\R_+) \to \mathfrak C^+(M,\xi) 
\ee
given by 
$$
\CI_{\lambda_0}(f) : =  f \lambda_0
 $$
 which depends on the choice of the reference form $\lambda_0$. Likewise, the set 
 $C^\infty(M,\R_+)$ is not a linear space, but its tangent space can be canonically
 identified with  $C^\infty(M,\R)$ via the map 
 $$
 h \mapsto h \lambda, \quad h := \frac{d f_t}{dt}\Big|_{t=0}
 $$
 for the germ of path $f_t$ with $f_0 = f$.
 We consider the composition
 $$
\Upsilon^{\text{\rm sm}}_{\lambda_0}: = 
\Upsilon^{\text{\rm sm}}_\xi \circ \CI_{\lambda_0}: C^\infty(M,\R_+) \to \Omega^1(M)
$$ 
so that we have the commutative diagram
$$
\xymatrix{C^\infty(M,\R_+) \ar[r]^{\Upsilon^{\text{\rm sm}}_{\lambda_0}}\ar[d]_{\CI_{\lambda_0}}  
& \Omega^1(M) \\
\mathfrak{C}(M,\xi) \ar[r]_{\iota_\xi} \ar[ur]_{\Upsilon^{\text{\rm sm}}_\xi}& \mathfrak{C}(M) \ar[u]_{\Upsilon^{\text{\rm big}}}
}
$$
We have the equalities
$$
\Upsilon^{\text{\rm sm}}_\xi = \Upsilon^{\text{\rm big}} \circ \iota_\xi
= \Upsilon^{\text{\rm sm}}_{\lambda_0}\circ  (\CI_{\lambda_0})^{-1}.
$$
In this way, we can remove $\lambda_0$-dependence in the description of of the tangent space 
$T_\lambda \mathfrak{C}^{\text{\rm sm}}(M,\xi)$ 
at $\lambda \in \mathfrak{C}^{\text{\rm sm}}(M,\xi)$.

\section{Moment maps and symplectic reductions}

The discussion about the \emph{infinite dimensional}
manifold $T^*\mathfrak{C}^+(M)$ in the present section and the next few
is largely in the formal level. However for those who feel uneasy of
this formal discussion, there is a way of making the discussion of this section rigorous 
by the procedure described in the following remark.

\begin{rem}\label{rem:strengthification} There is a canonical way of defining the weak symplectic manifold 
$(T^*\mathfrak C^+(M), \omega_0)$ by considering the weaker notion of  `dual space'  consisting of \emph{bounded
linear functionals} instead of `\emph{continuous functionals} in the sense of
\emph{locally convex topological vector spaces}, and this is described in
\cite[Section 48.2]{kriegl-michor}, for example. 
In this weaker notion of dual space, the model space of $\mathfrak{C}^+(M)$ 
is reflexive since it is a Fr\'echet \emph{Montel} space. 
(We refer to \cite[Results 6.5 \& 52.26]{kriegl-michor} for some relevant discussion and other
additional materials.)
Thanks to the reflexivity of $\mathfrak{C}^+(M)$ in this sense, $\omega_0$ becomes a (strong) symplectic structure on $T^*\mathfrak C^+(M)$ by \cite[Section 48.3.]{kriegl-michor}. 
 But we prefer not to talk much about these subtle treatments of 
 weak symplectic structure here but postpone elsewhere,  
 because the associated functional analysis issues 
 will dilute the main theme of the paper.
\end{rem}

\begin{defn}
Let $(P, \omega)$ be a symplectic manifold and $\Phi: G \times P \rightarrow $ a symplectic action of a (finite dimensional) Lie group $G$ on $P$; that is, for each $g \in G$, the map $\Phi_g: P \rightarrow P$ is symplectic. We say that a (smooth) map
$$
\mathcal{J} : P \rightarrow \mathfrak{g}^* \; (\text{the dual of the Lie algebra}\, \mathfrak{g} \,  \text{of} \,G)
$$
is a \textbf{moment map} if for every $\zeta \in \mathfrak{g}$, 
$$
d\langle \CJ, \zeta\rangle=\zeta_P \intprod \omega
$$
where $\langle \CJ, \zeta \rangle$ is an $\R$-valued function
defined via the canonical pairing between $\mathfrak g$ and $\mathfrak g^*$.
By definition, its Hamiltonian vector field is the infinitesimal generator 
 $\zeta_P$ corresponding to $\zeta$ under the action  $\Phi$.
 \end{defn}
 
Recall  that not every symplectic action has a moment map.

\begin{defn}
Let $\Phi: G \times P \to P$ be a symplectic action of $G$ on $P$. 
Assume that $\mathcal{J}: P \rightarrow \mathfrak{g}^*$ is a moment map for $\Phi$.\\
We say $\mathcal{J}$ is an $\text{Ad}^*$-equivariant moment map for the action $\Phi$ if 
$$
\langle \mathcal{J}(\Phi_g(x)),\zeta \rangle 
= \langle \mathcal{J}(x),\text{Ad}_{g^{-1}} \zeta \rangle 
=: \langle \text{Ad}_{g}^*\,\mathcal{J}(x) , \zeta \rangle
$$
for every $g\in G$, $x \in P$, and $\zeta \in \mathfrak{g}$.
\end{defn}

Finally, we recall the \textit{symplectic reduction} \cite{marsden-weinstein} which is usually 
applied to the finite dimensional case, but its formal construction can be applied to
the infinite dimensional case, ignoring the functional analytical issues arising from
the infinite dimensional situation.

\begin{thm}[Marsden-Weinstein]
Let $(P,\omega)$ be a symplectic manifold on which a (finite dimensional) Lie group $G$ acts symplectically and let $\mathcal{J}: P \rightarrow \mathfrak{g}^*$ be an $\text{Ad}^*$-equivariant moment map for this action. Assume $\mu \in \mathfrak{g}^*$ is a regular value of $\mathcal{J}$ and that the isotropy group $G_{\mu}:=\{g \in G : \text{Ad}_g^*\, \mu = \mu\}$ acts freely and properly on $\mathcal{J}^{-1}(\mu)(\neq \emptyset)$. Then $P_\mu:=\mathcal{J}^{-1}(\mu)/G_{\mu}$ has a unique symplectic form $\omega_{\mu}$ with the property
$$
\phi_{\mu}^* \omega_{\mu}=i_{\mu}^* \omega
$$
where $\phi_{\mu}: \mathcal{J}^{-1}(\mu) \rightarrow P_\mu$ is the canonical projection and $i_{\mu}: \mathcal{J}^{-1}(\mu) \rightarrow P$ is the inclusion.
\end{thm}

\section{Multi-Hamiltonian system and relative information entropy}
\label{sec:family}

We adopt the following terminologies coming from statistical mechanics
as in \cite{lim-oh}.

\begin{defn} (Observables and observations)
We call a (smooth) function $F: M \rightarrow \mathbb{R}$ a $\textit{local observable}$
 and a collection of (smooth) functions 
$$
\mathcal{F}=\{F_1, \ldots , F_N\}
$$
a $\textit{local observable system}$.
\begin{enumerate}
\item  Each $\lambda \in \mathfrak{C}^+(M)$ with respect to which $F$ is an $L^1(\mu_{\lambda})$-function defines an $\textit{observation}$ of the local observable $F$ relative to $\lambda$ 
$$
\mathcal{O}_{F}(\lambda):=\int_M F d\mu_{\lambda}.
$$
\item We define a $\textit{collective}$ to be the map $\mathcal{O}_{\mathcal{F}}: \mathfrak{C}^+(M) \rightarrow \mathbb{R}^N$ given by 
$$
\mathcal{O}_{\mathcal{F}}=(\mathcal{O}_{F_1}, \ldots, \mathcal{O}_{F_N} ).
$$ 
\end{enumerate}
We call $\mathcal F$ a \emph{contact multi-Hamiltonian system} when a background
contact form $\lambda$ is equipped with $M$.
\end{defn}
We recall from \cite{lim-oh} that 
the relative information entropy $\mathcal{S}$ is  defined to be the observation 
associated to the $\textit{relative information density}$, which is $\textit{intrinsic}$ 
and $\textit{universal}$ in the sense that $\mathcal{S}$ 
\emph{depends only on the contact form $\lambda$ independent of other local observables}.

Recall the definition of the Kullback-Leibler divergence  in the information theory \cite{kullback-leibler}
is given by
\be\label{eq:DKL}
D_{\text{\rm KL}}(\rho'|\rho) = \int_M \log \frac{d\rho'}{d\rho}\, d\rho'.
\ee
We can apply this definition to the measure of the form $\rho = \mu_\lambda$.
In the present paper, we consider $\lambda \in \EC^+(M)$ not just those normalized
 $\lambda \in \EC_1(M)$ with no essential change of its definition, although 
$D_{KL}(\rho' | \rho )$ is usually defined for probability measures $\rho', \rho$.

\begin{defn} [Contact relative entropy] Let $\lambda, \, \lambda' \in \mathfrak{C}^+(M)$ and let
$\varphi \in \Diff^+(M)$ be an orientation-preserving diffeomorphism relative to the
orientation induced by the coorientation of $\lambda$.
\begin{enumerate}
\item 
We define the \emph{contact relative entropy} of $\lambda'$
relative to  $\lambda$ to be
$$
\mathcal{S}(\lambda' | \lambda):= D_{\text{\rm KL}}(\mu_{\lambda'}|\mu_\lambda)
$$
We denote by $\CS_\lambda(\cdot):=  \CS (\cdot | \lambda): \mathfrak{C}^+(M) \to \R$ the associated 
entropy function (relative to $\lambda)$.
\item We define the \emph{contact $\lambda$-entropy of $\varphi$} to be
$$
\mathcal{H}_{\lambda}(\varphi):=  \mathcal S(\varphi^*\lambda|\lambda) = 
\int_M  \log \left(\frac{d \mu_{\varphi^*\lambda}}{d \mu_\lambda}\right)\,
d\mu_{\varphi^*\lambda}. 
$$    
\end{enumerate}
\end{defn}
For the later purpose, we also name the assignment $\lambda \mapsto \lambda \wedge (d\lambda)^n$ 
\be\label{eq:vol}
\vol(\lambda) = \mu_\lambda.
\ee
Then we have $\CS = D_{\text{\rm KL}} \circ (\vol \times \vol)$ and $\CS_{\lambda} = D_{\text{\rm KL}}^{\mu_{\lambda}} \circ \vol$
where we put
\be\label{eq:Slambda0}
D_{\text{\rm KL}}^{\mu_{\lambda}}(\rho) = D_{\text{\rm KL}}(\rho|\mu_{\lambda}).
\ee
We mention that $D_{\text{\rm KL}}^{\mu_{\lambda}} $ is a function defined on the space of positive volume forms $\CD^+(M)$.\\

In fact, the following proposition shows that $\CS_{\lambda}$ is invariant under 
the action of the group $\Diff (M, \mu_{\lambda})$ consisting of $\mu_{\lambda}$-measure preserving 
diffeomorphisms.

\begin{prop} \label{prop:invariance-of-SS}
Let  $\lambda, \, \lambda' \in \mathfrak {C}^+(M)$. Then the following hold:
\begin{enumerate}
\item Suppose $V(\lambda) = V(\lambda') = V < \infty$. Then, the contact relative entropy $\mathcal{S}(\lambda'|\lambda)$ is nonnegative.
\item  $\CS_{\lambda}$ is invariant under the action of $\mu_\lambda$-preserving diffeomorphisms 
of $M$. 
\end{enumerate}
\end{prop}
\begin{proof} The first statement follows from Jensen's inequality because $-\log x$ is a
convex function: Knowing that both $\frac1V \mu_\lambda$ and $\frac1V \mu_\lambda'$
are probability measures, we derive
\beastar
- \CS_{\lambda}(\lambda') & = & 
V \int_M - \log\left(\frac{d\mu_{\lambda'}}{d\mu_{\lambda}}\right) \frac{d\mu_{\lambda'}}{V} \\
& \leq  & V  \log \left(\int_M \frac{d\mu_{\lambda}}{d\mu_{\lambda'}}\,\frac{ d\mu_{\lambda'}} {V}\right)
= V \log \left(\int_M \frac{d\mu_{\lambda}}{V} \right)\\
& = & V \log 1 = 0.
\eeastar

For the second statement, we utilize the equality from Lemma \ref{lem:diff-actionon-mu}.
We compute
\beastar
\CS_{\lambda}(\varphi_*\lambda') &=&\int_{M} \log 
\left(\frac{d\mu_{\varphi_*\lambda'}}{d\mu_{\lambda}}\right)d\mu_{\varphi_*\lambda'}
=\int_{M} \log \left(\frac{d(\varphi_* \mu_{\lambda'})}{d\mu_{\lambda}} \right) d(\varphi_* \mu_{\lambda'})\\
& = & \int_{M} \log \left(\frac{d(\varphi_* \mu_{\lambda'})}{d(\varphi_* \mu_{\lambda})} \right) d(\varphi_* \mu_{\lambda'})
= \int_{M} \log \left(\frac{d\mu_{\lambda'}}{d\mu_{\lambda}} \right) \, d\mu_{\lambda'} = 
 \CS_{\lambda}(\lambda')
\eeastar
where we use $\varphi_*\mu_{\lambda'} = \mu_{\varphi_*{\lambda'}}$ for the second equality and
the assumption $\varphi_*\mu_\lambda = \mu_\lambda$ for the third and the change of variable for the 
fourth equality: More specifically, we have
$$
\frac{d(\varphi_* \mu_{\lambda'})}{d(\varphi_* \mu_{\lambda})} =  
\varphi_*\left(\frac{d\mu_{\lambda'}}{d\mu_{\lambda}} \right)
$$
which implies the fourth equality.
\end{proof}
Motivated by this proposition, we introduce the following subset consisting of 
$\mu_\lambda$-measure preserving diffeomorphisms.

\begin{defn}[$\lambda$-incompressible diffeomorphisms] Let $\lambda \in \mathfrak{C}^+(M)$.
We call a $\mu_\lambda$-preserving diffeomorphisms a 
\emph{$\lambda$-incompressible diffeomorphism}, and denote the set by
$$
\Diff(M,\mu_\lambda) = \{\varphi \in \Diff^+(M) \mid \varphi^*\mu_\lambda = \mu_\lambda\}.
$$
\end{defn}

\begin{rem}\label{rem:small-CS}
We mention that the restriction of the relative entropy function $\CS_\lambda$ naturally 
restricts to s-CKTPS 
$$
\CS_\lambda^{\text{\rm sm}}: \mathfrak{C}^+(M,\xi) \to \R, \quad \xi=\ker \lambda
$$
which is also invariant under the action of $\Diff(M,\mu_\lambda) \cap \Cont_+(M,\xi)=$\\  $\Cont^{\text{st}}(M, \lambda)$, where $\Cont^{\text{st}}(M, \lambda)=\{\phi \in \Diff(M) \mid \phi^*\lambda=\lambda \}$.
\end{rem}

\section{Observables and regularity of observation map}
\label{sec:regularity}

To carry out the Marsden-Weinstein reduction even in the formal level, we should 
apply the reduction to the nonempty preimage of a \emph{regular value} of $\CO_\CF$. 

\subsection{Some contact Hamiltonian calculus}

Since we restrict the measures to those
$\mu_\lambda$ arising from contact forms $\lambda$, we need some preparation 
for some contact Hamiltonian calculus which is a continuation of 
\cite{dMV}, \cite{BCT} and \cite{oh:contacton-Legendrian-bdy}.

Recall \eqref{eq:defining-Rlambda} induces the decomposition
\be\label{eq:TM-decompose}
TM = \ker \lambda \oplus \R\langle R_\lambda \rangle
\ee
and hence its dual decomposition
\be\label{eq:T*M-decompose}
T^*M = (\langle R_\lambda \rangle)^\perp \oplus \xi^\perp
\ee
under which we decompose $\alpha = \alpha^\pi + \alpha(R_\lambda) \lambda$.
It is easy to see that the part $\alpha^\pi$ can be written as $\alpha^\pi = Y^\pi \intprod d\lambda$ 
for a unique  vector field $Y^\pi$ tangent to $\xi$ by the nondegeneracy of $d\lambda$ on $\xi$.

We also recall the definition of contact Hamiltonian vector field $X_F^\lambda$
whose defining equation is given by
\be\label{eq:XF-defining}
\begin{cases}
X_F^\lambda \intprod \lambda = - F\\
X_F^\lambda \intprod d\lambda = dF - R_\lambda[F] \lambda
\end{cases}
\ee
according to the sign convention of \cite{dMV}, \cite{oh:contacton-Legendrian-bdy}.
In other words, we have the decomposition 
$$
X_F^\lambda = (X_F^{\lambda})^{\pi} - F R_\lambda
$$
where the $\xi$-component $(X_F^{\lambda})^{\pi}$ is uniquely determined by the equation\\ 
$(X_F^{\lambda})^{\pi} \intprod d\lambda = dF - R_\lambda[F] \lambda$.

More generally, we have the following one-to-one correspondence between the
vector fields and the one-forms on contact manifolds in terms of the
decomposition \eqref{eq:TM-decompose} and \eqref{eq:T*M-decompose}.

\begin{lem} \label{lem:decomposition}
Let $(M,\xi)$ be equipped with a contact form $\lambda$. Then
the following holds under the the
decomposition \eqref{eq:TM-decompose} and \eqref{eq:T*M-decompose}.
\begin{enumerate}
\item Any one-form $\alpha$ can be uniquely decomposed into
$$
\alpha = h_\alpha \lambda + Y^\pi_\alpha \intprod d\lambda.
$$
When $\alpha$ is exact, i.e., $\alpha = dg$ for some real-valued function $g$, the correspondence 
becomes $dg =  R_\lambda[g]\, \lambda + X_g^\pi \intprod d\lambda$, i.e.,
$$
h_{dg} = R_\lambda[g], \quad Y^\pi_{dg} = X_g^\pi.
$$
\item  Conversely any vector field $X$ can be decomposed into
$$
X = X^\pi + \lambda(X) R_\lambda.
$$
When $X$ is a contact vector field, i.e. $X = X_h$ for some function $h$,
then we have $\lambda(X) = -h$.
\end{enumerate}
\end{lem}

Next we study the effect of the first variation of contact forms on the associated volume
form $\mu_\lambda$. 
First, we recall the definition of volume form $\mu_\lambda$ 
$$
\mu_\lambda = \lambda \wedge (d\lambda)^n.
$$
We consider the function $\vol: \mathfrak C^+(M) \to \Omega^{2n+1}(M)$ given by
$$
\vol(\lambda) = \mu_\lambda.
$$
The following formula will be used later in the paper.

\begin{prop}\label{prop:dvol} Let $\lambda$ be given and take an adapted 
CR-almost complex structure $J$ and consider metric $g$ 
$$
g = d\lambda(\cdot, J\cdot) + \lambda \otimes \lambda
$$
associated to the contact triad $(M,\lambda, J)$. Decompose
$\alpha = h_\alpha \lambda + Y^\pi_\alpha \intprod d\lambda$ as in Lemma \ref{lem:decomposition}.
 Then we have
\be\label{eq:dvol}
d\vol(\lambda)(\alpha) =  ((n+1) h_\alpha + \nabla \cdot Y_\alpha^\pi) \, \mu_\lambda.
\ee
\end{prop}
\begin{proof} We compute
\beastar
d\vol(\lambda)(\alpha) & = & \alpha \wedge (d\lambda)^n + n \lambda \wedge d\alpha \wedge (d\lambda)^{n-1}\\
& = & h_\alpha \lambda \wedge (d\lambda)^n + (Y^\pi_\alpha \intprod d\lambda) \wedge (d\lambda)^n
\nonumber \\
&{}& + n \lambda \wedge d (Y^\pi_\alpha \intprod d\lambda) \wedge (d\lambda)^{n-1}+ n \lambda \wedge (h_{\alpha} d\lambda) \wedge (d\lambda)^{n-1} \\
& = & (n+1) h_\alpha \mu_\lambda + (Y^\pi_\alpha \intprod d\lambda) \wedge (d\lambda)^n
+ n \lambda \wedge  \CL_{Y^\pi_\alpha} d\lambda \wedge (d\lambda)^{n-1} \\
& = & (n+1) h_\alpha \mu_\lambda + (Y^\pi_\alpha \intprod d\lambda) \wedge (d\lambda)^n
 + \lambda \wedge  \CL_{Y^\pi_\alpha}  (d\lambda)^n \\
& = & (n+1) h_\alpha \mu_\lambda + (Y^\pi_\alpha \intprod d\lambda) \wedge (d\lambda)^n \\
&{}& + \CL_{Y^\pi_\alpha} (\lambda \wedge (d\lambda)^n)
- (\CL_{Y^\pi_\alpha} \lambda)\wedge (d\lambda)^n \\
& = & (n+1) h_\alpha \mu_\lambda 
+ \CL_{Y^\pi_\alpha} (\lambda \wedge  (d\lambda)^n) \\
& = & ((n+1) h_\alpha + \nabla \cdot Y_\alpha^\pi) \, \mu_\lambda.
\eeastar
\end{proof}
For the simplicity of notation, we will often write 
$$
d\vol(\lambda)(\alpha) =: \delta_\alpha \mu_\lambda.
$$

\subsection{A regularity criterion of the observation on b-CKTPS }\hspace{5pt}

For the  purpose of reduction of big phase space, 
we recall that the set of contact forms is an open subset of the set of one-forms on $M$,
if there is one on $M$. Therefore we have the identification
$$
T_\lambda \mathfrak{C}^+(M) \cong \Omega^1(M).
$$
We can regard $\alpha \in \Omega^1(M)$ as a tangent vector in 
$T_\lambda \mathfrak{C}^+(M)$. 
And we study the submersion property via the standard method of the Fredholm alternative.

In this regard, we prove the following proposition which was mentioned in \cite{lim-oh}
in the general context of probability measures. 

\begin{prop}\label{prop:regularity} Suppose that $\CF$ has the property that 
the set $\{F_1, \cdots, F_N\}$ is linearly independent. Then 
$\CO_\CF$ is a submersion, i.e., $d\CO_\CF(\lambda)$ is surjective at all $\lambda$.
\end{prop}
\begin{proof}
Let $\lambda \in \mathfrak{C}^+(M)$ and $F \in C^\infty(M,\R)$ be given and $\alpha = h_\alpha \, \lambda + Y^\pi_\alpha \intprod d\lambda$.
Consider the derivative
$$
d\CO_\CF (\lambda): T_\lambda \mathfrak{C}^{+}(M) \to \R.
$$
identifying $T_\lambda \mathfrak{C}^+(M)$ with $\Omega^1(M)$. Then we have
\bea\label{eq:dCOCF}
d\CO_F (\lambda)(\alpha) & = & \int_M F d\vol(\lambda)(\alpha) \nonumber \\
& = & \int_M F \left ((n+1) h_\alpha \right)\, \mu_\lambda
\eea
from Proposition \ref{prop:dvol}.
This implies that the $L^2$-cokernel $(c_1, \cdots, c_N)$ is characterized by
$$
\sum_{i=1}^N c_i \left(\int_M  F_i \left((n+1) h_{\alpha} ) \right) d\mu_\lambda\right) = 0
$$
for all $h_{\alpha}$ and $Y_{\alpha}^\pi$ whose choices one-to-one corresponds to that of 
$\alpha = h_\alpha \lambda + Y^\pi_\alpha \intprod d\lambda$. 
Therefore just by considering
$\alpha$ with $Y^\pi_\alpha = 0$, the equation is reduced to
$$
\int_M \left(\sum_{i=1}^N c_i F_i  \right)h \, d\mu_\lambda = 0
$$
with $h \in C^\infty(M,\R)$ for all $h$. This then implies $\sum_{i=1}^N c_i F_i = 0$ 
which implies $c_i = 0$ for all $i$.
\end{proof}

\subsection{A regularity criterion of the observation on s-CKTPS}\hspace{5pt}

The following is a universal regularity criterion for the restriction of observation map 
$\CO_\CF$ to $\mathfrak C^+(M, \xi)$
$$\CO_\CF^{\text{\rm sm}}: \EC^+(M, \xi) \rightarrow \R^N. 
$$

\begin{prop}\label{prop:regularity-small} Let $\CF = \{F_1, \cdots, F_N\}$ be an observable system that is 
linearly independent as elements of $C^\infty(M,\R)$. 
Then the map
$$
\CO_\CF^{\text{\rm sm}}: \mathfrak{C}^+(M,\xi) \to \R^N
$$
is a submersion.
\end{prop}
\begin{proof}
Fix a reference contact form $\lambda_0$ and
identify $\mathfrak{C}^+(M,\xi) \cong C^\infty(M,\R_{+} )$ by writing $\lambda = f \lambda_0$ for $f > 0$.
Then we have
$$
\CO_F^{\text{\rm sm}}(\lambda) = \int_M F \, d\mu_\lambda.
$$
Then the first variation in the small phase can be written as
$$
\delta \lambda = \alpha = h\, \lambda, \quad  h \in C^\infty(M,\R)
$$
and
$$
d\CO_{F_i}^{\text{\rm sm}}(\lambda)(\alpha) =
(n+1) \int_M F_i  h \, \mu_\lambda
$$
for $i = 1, \cdots, N$, which corresponds to the case $Y^\pi_\alpha = 0$ in (\ref{eq:dCOCF}).

We then obtain the characterization of the $L^2$-cokernel element $(c_1, \cdots, c_N)$ of $d\CO_{\CF}^{\text{\rm sm}}$
\beastar
0 & = & \sum_{i=1}^N c_i (n+1) \int_M F_i  h \lambda \wedge (d \lambda)^n\\
& = & \int_M \sum_{i=1}^N (n+1) c_i F_i  h\,  d \mu_{\lambda}
= \int_M \left( \sum_{i=1}^N (n+1) c_i F_i \right) h\, d \mu_{\lambda}
\eeastar
for all $h \in C^\infty(M, \R)$. Therefore we have obtained
$$
\sum_{i=1}^N  c_i F_i = 0.
$$
 By the standing hypothesis of linear independence of $\{F_1, \cdots, F_n\}$,
we have derived $c_i = 0$ for all $i$, and hence $\coker d\CO_{\CF}^{\text{\rm sm}}(\lambda) = \{0\}$ 
at all $\lambda
\in \mathfrak C^+(M,\xi)$. This proves that $\CO_{\CF
}^{\text{\rm sm}}$ is a submersion.
\end{proof}

\section{Reduction of CKTPS}
\label{sec:reduction}

We now perform the first stage of thermodynamic reduction by decomposing 
CKTPS with those of iso-data over the collective observation data by applying 
Marsden-Weinstein reduction.

\begin{defn} Let $\pi : T^* \mathfrak C^+(M) \rightarrow \mathfrak C^+(M) $ be the natural projection.
Then we lift each function $\mathcal{O}_{F_i}$ to $\widetilde{\mathcal{O}}_{F_i}=\mathcal{O}_{F_i} \circ \pi$ defined on
 $T^* \mathfrak C^+(M)$ and we call it the \textit{the lifted observation} of $F_i$. 
 Similarly we write $\widetilde{\mathcal{S}}_{\lambda_0}=\mathcal{S}_{\lambda_0} \circ \pi$ and call it the \textit{lifted relative information entropy}.
\end{defn}

We represent an element of $T^*\mathfrak C^+(M)$ by $(\lambda,\beta^\flat)$. Then
we have
$$
\omega_0((\alpha,X) ,(\alpha',X'))=\langle X', \alpha \rangle  - \langle  X, \alpha' \rangle
$$
 for each pair
$$
(\alpha,X) , (\alpha', X') \in T_{(\lambda,\beta^\flat)}(T^* \mathfrak {C}^+(M)) \cong 
\Omega^1(M) \oplus \Omega^1(M)^*,
$$
with respect to the natural pairing $\langle \cdot, \cdot \rangle$ between $\Omega^1(M)^*$ and
$\Omega^1(M)$.
Under the regularity conditions of the observation map such as the one given in Proposition \ref{prop:regularity}
or in Proposition \ref{prop:regularity-small}, $\{X_{\widetilde{\mathcal{O}}_{F_1}}, \ldots , X_{\widetilde{\mathcal{O}}_{F_N}}\}$ 
is linearly independent.

It is a standard fact that the vector fields are tangent to the 
fibers of $T^* \mathfrak C^+(M)$ and induces a (global) Hamiltonian flow $\phi_i^t$ with $(\phi_i^t)^*\omega_0=\omega_0$
which is nothing but the linear translation along the fiber by the one-form
$$
d\mathcal{O}_{F_i}.
$$ 
Furthermore they Poisson-commute and so define a (symplectic) $\mathbb{R}^N$-action on $T^* \mathfrak C^+(M)$ by
\bea
&\mathbb{R}^N \times T^* \mathfrak C^+(M) \; \rightarrow \; T^* \mathfrak C^+(M)\\
&((t_1,\ldots , t_N) , (\lambda,\beta^\flat)) \mapsto 
(\phi_1^{t_1}\circ \ldots \circ \phi_N^{t_N})(\lambda,\beta^\flat).
\eea
Moreover, when the regularity hypothesis of $\CF$ such as one given in 
Proposition \ref{prop:regularity} is satisfied, we can identify 
$$
\CG_{\CF}:=\{\phi_1^{t_1}\circ \ldots \circ \phi_N^{t_N} : (t_1, \ldots , t_N) \in \mathbb{R}^N \} \cong \R^N.
$$

Finally, we recall the following standard fact (See \cite{abraham-marsden}.)
The moment map 
$$
\mathcal{J}_{\mathcal{F}} : T^* \mathfrak C^+(M) \rightarrow \mathfrak{g}_{\mathcal{F}}^*
$$
of the action of $\CG_\CF$ on $T^* \EC^+(M)$
 is characterized by the formula
$$
\langle \mathcal{J}_{\mathcal{F}}(\lambda,\beta^\flat), X_{{\widetilde{\mathcal{O}}_{F_i}}} \rangle 
=\mathcal{O}_{F_i}(\lambda), \quad  i=1, \ldots, N
$$
for $(\lambda,\beta^\flat) \in T^* \mathfrak C^+(M)$. It is also an $Ad^*$-equivalent moment map for
$\CG_{\CF}$.  

Therefore we have obtained the reduced space
$$
\CJ_{\CF}^{-1}(\mu)/\CG_{\CF}
$$
for each regular value $\mu$ of the observation map $\CO_\CF$.

At this stage, we can repeat the process performed in 
\cite[Section 5 \& 6]{lim-oh} which we duplicate in the next two sections adapted to
the current contact phase space case.

\section{A generating function of contact thermodynamic equilibrium}
\label{sec:reduced-entropy}

We start with the following definition
  
\begin{defn}[Observation data set] We denote by $B_\CF$ the image of the moment map $\CJ_\CF$
and by $B^{\circ}_\CF$ the set of its regular values. We call $B_\CF$ the \emph{observation data set}.
\end{defn}
In the circumstances where  $\CO_{\CF}$ holds is a submersion e.g., under 
the regularity condition as in Proposition \ref{prop:regularity} or in Proposition \ref{prop:regularity-small}, 
$B^{\circ}_\CF=B_\CF$ by the definition, which we will assume henceforth.

Then we have decomposition
$$
T^*\EC^+(M) = \bigcup_{\mu \in B_\CF} \{\mu\} \times \CJ_\CF^{-1}(\mu)
$$
and
$$
T^*\EC^+(M)/\CG_\CF = \bigcup_{\mu \in B_\CF} \{\mu\} \times \CJ_\CF^{-1}(\mu)/\CG_\CF.
$$
We attract readers' attention that the coadjoint isotropy group
of $\CG_\CF$  is the full group
\begin{equation}\label{eq:isotropy-group}
\CG_{\CF,\mu} = \CG_\CF
\end{equation}
for all $\mu$ and hence the reduced space at $\mu$ becomes
$$
\CJ_\CF^{-1}(\mu) /\CG_\CF = : \CM^{\CF}_{\mu}
$$
for all regular values of $\mu$. Obviously $\CJ_\CF^{-1}(\mu) = \emptyset$ unless
$\mu \in B_\CF$.

We summarize the above discussion into
\begin{cor} Let $ \mu = (\mu_1, \cdots, \mu_N)$ be a collective observation of
$\CF = \{F_1, \cdots, F_N\}$. Then the reduced space denoted by
\begin{equation}\label{eq:MmuF}
\CM_{\mu}^{\CF} : = \CJ_\CF^{-1}(\mu) /\CG_\CF\end{equation}
is a (infinite dimensional) symplectic manifold.
\end{cor}

The projection $\CM^{\CF} \to B^{\circ}_\CF \subset {\mathfrak g}_{\CF}^*$ forms 
a \emph{symplectic fiber bundle} over
$B^{\circ}_\CF\subset \mathfrak g_{\CF}^*$. This provides a \emph{family of functions} 
on $B_\CF$
\begin{defn} 
We call the union
$$
\CM^{\CF} : = \bigsqcup_{\mu \in B^{\circ}_\CF}  \{\mu\} \times \CM^{\CF}_{\mu};
\quad \mu: = (\mu_1,\cdots, \mu_n)
$$
the \emph{$\CF$-reduced contact kinetic theory phase space} ($\CF$-reduced $\text{\rm CKTPS}$)
associated to $\CF$,
where $\CJ = \CJ_\CF$ is the moment map associated to the symmetry group
generated by the induced Hamiltonian flows on $T^*\EC^+(M)$.
\end{defn}

Recall that $\widetilde S_{\lambda_0}$ just descends to $\CJ_{\CF}^{-1}(\mu)/\CG_{\CF}$ 
because $\widetilde S_{\lambda_0}$ is independent on the fibers of $T^* \EC^+(M)$ and $G_\CF$ 
acts by fiberwise translations on the fibers of $T^* \EC^+(M)$. (See Section \ref{sec:reduction}.)

\begin{cor} [$\CF$-reduced  entropy]\label{cor:universality}
The lifted relative information entropy $\widetilde \CS_{\lambda_0} = \CS_{\lambda_0} \circ \pi: T^*\EC^+(M) \to \R$
 is universally reduced to a well-defined function
 $$
 \CS_\CF^{\text{\rm red}}: \CM^{\CF} \to \R.
 $$
 We call $\CS^{\text{\rm red}}_\CF$ the \emph{$\CF$-reduced entropy function}.
\end{cor}

This construction provides a natural \emph{family of functions} over $B_\CF$, i.e.,
which we will show generates a Legendrian submanifold of $J^1B_\CF$ via the method of
\emph{generating functions}

%
%

By definition, we have the obvious commutative diagram
\be \label{eq:hatS}
\xymatrix{ \CM^\CF \ar[d]_{\pi} \ar[dr]^>>>>>>>>>{\CS_\CF^{\text{\rm red}}} \\
{\EC^+(M)}  \ar[d]_{\CO_\CF} \ar[r]_>>>>>>{\CS_{\lambda_0}}& \R
\\
{B_{\CF}}  &}.
\ee
Using this commutative diagram, 
it is easy to see that both families of $\CS_\CF^{\text{\rm red}}$ and
of $\CS_{\lambda_0}$ generate the same Legendrian submanifold by the canonical procedure 
of the generation of a Legendrian submanifold via the generating functions. 
Therefore we instead 
use the fibration $\CO_{\CF} : \EC^+(M) \rightarrow B_{\CF} \subseteq \R^N$ 
and a generating function $\CS_{\lambda_0}$ defined on $\mathfrak{C}^+(M)$.

In this setting, the fiber 
at $q \in B_\CF$ of the \emph{vertical critical set} of $\CS_{\lambda_0}$ is nothing but the intersection
$$
(d^v\CS_{\lambda_0})^{-1}(0) \cap \CO_\CF^{-1}(q).
$$
(See the local coordinate description the Lagrange multiplier in \cite[Section 8]{lim-oh}
or Section \ref{sec:description}.)

The procedure of finding a critical point of $\CS_{\lambda_0}$
can be decomposed into the two steps. We first choose
an Ehresmann connection
\be\label{eq:VH-splitting}
T\mathfrak{C}^+(M) = VT \mathfrak{C}^+(M) \oplus HT\mathfrak{C}^+(M)
\ee
of the fibration  $\CO_\CF: \mathfrak{C}^+(M) \to B_\CF$,
and decompose the differential $d\CS_{\lambda_0}(\lambda)$
into the vertical and the horizontal components 
$$
d\CS_{\lambda_0}(\lambda) = d^v\CS_{\lambda_0}(\lambda)
+ D^h\CS_{\lambda_0}(\lambda)
$$
with respect to the splitting \eqref{eq:VH-splitting}.
Recall that the vertical differential is canonically defined but the
horizontal differential of $\CS_\CF^{\text{\rm red}}$ needs the use of connection. 
In general the horizontal component $D^h\CS (\lambda)$ depends on the
connection but it will be independent thereof at the vertical critical
point $\lambda$ where $d^v\CS(\lambda) = 0$. 

Applying the standard generating function construction,
we consider the vertical critical set
$$
\Sigma_{\CS_{\lambda_0};\CF}: = \{\lambda\in {\mathfrak{C}^+(M)} \mid d^v\CS_{\lambda_0}(\lambda) = 0\}.
$$
This would become a smooth submanifold of $\mathfrak{C}^+(M)$ 
\emph{provided the standard transversality hypothesis}
\be\label{eq:transversality-dvCSCF}
d^v \CS_{\lambda_0} \pitchfork 0_{ (VT)^*\mathfrak{C}^+(M)}
\ee
should hold where $0_{ (VT)^*\mathfrak{C}^+(M)}$ is the zero section 
of $(VT)^*\mathfrak{C}^+(M) \to \mathfrak{C}^+(M)$.  For the convenience of
our exposition, we introduce the following definition where \emph{we will be somewhat
formal and  vague about the precise  description of smoothness in the setting 
of infinite dimensional situation as in the case of our current interest. To make
this discussion completely rigorous, one needs to take a suitable Banach
completion or to take the convenient setting of \cite{kriegl-michor}, whose discussion we
omit here but postpone elsewhere.}

\begin{defn}[Morse family]\label{defn:relative-Morse}
 Consider a smooth fibration $\pi: E \to B$ and a smooth function $S : E \to \R$
visualized in the diagram
\begin{equation}\label{eq:relative-Morse}
\xymatrix{E  \ar[d]^{\pi} \ar[r]^S & \R \\
B &
}
\end{equation}
We say $S$ is \emph{$\pi$-relative Morse}  with respect to the fibration $\pi: E \to B$,
if it satisfies \eqref{eq:transversality-dvCSCF}, or simply \emph{relative Morse} when no 
explicitly mentioning of the fibration is needed.
 \end{defn}
The upshot of introducing such a definition is first to ease our exposition, and
also to connect the notion of generating function 
to the setting that has recently attracted much attention in the context of derived algebraic
geometry. (See \cite{anel-calaque}, for example.) With this definition, one can restate the
method of generating functions as
`\emph{Any Morse family generates a Legendrian submanifold of the 1-jet bundle $J^1B$ of the 
base $B$}' especially when the base has finite dimension.

Now the map $\iota_{\CS_{\lambda_0}}: \Sigma_{\CS_{\lambda_0};\CF} \to J^1B_\CF$
given by
\be\label{eq:iota-CSCF} 
\iota_{\CS_{\lambda_0}}(\lambda) = \left(\CO_\CF(\lambda), D^h\CS_{\lambda_0}(\lambda),
\CS_{\lambda_0} (\lambda)\right )\in  J^1 \R^N
\ee
is well-defined. This leads us to the following definition.

\begin{defn}[Contact thermodynamic equilibrium] 
We call the subset of $J^1\R_N$ defined by
\be\label{eq:RCF}
R_{\CS_{\lambda_0};\CF} : = \Image \, \iota_{\CS_{\lambda_0}} \subset  J^1 \R^N \,
\ee
the \emph{contact thermodynamic equilibrium of a system $\CF$ on $M$.} 
\end{defn}

We will examine the transversality condition for the the multi-Hamiltonian system
$\CF=\{F_1, \cdots, F_N\}$ in Section \ref{sec:transversality}, and provide an
explicit class of such a system that satisfies the transversality \emph{ for the 
small phase space}, and then give the dimension formula and prove
the Legendrian property of $R_{\CS_{\lambda_0};\CF}$ in Appendix \ref{sec:Legendrian-property}.

We now put our construction in a categorical setting as follows.
Introduction of the following definition is useful for this purpose.

\begin{defn} For each nonnegative integer $N \in \N$, we consider the Cartesian power
$(C^\infty(M,\R))^N$. We define the subset
\beastar
\mathfrak{Disc}_N(M,\R) & : = & \{ (F_1, \cdots, F_N) \in (C^\infty(M,\R))^N \mid \\
&{}& \quad \{F_1, \cdots, F_N\} \, \text{\rm is not linearly independent}\}
\eeastar
and consider its complement denoted by
$$
(C^\infty(M,\R))^N \setminus \mathfrak{Disc}_N(M,\R) = : (C^\infty(M,\R))^N_{\text{\rm reg}}.
$$
\end{defn}
Here `$\mathfrak{Disc}$' stands for `\emph{discriminant}'. We note that the subset
$\mathfrak{Disc}_N(M,\R)$ is a stratified subvariety of 
$(C^\infty(M,\R))^N$ with  \emph{finite codimension}.
\begin{defn} Let $N$ and $(C^\infty(M,\R))^N_{\text{\rm reg}}$ be as above. We consider 
the universal observation map 
$$
\CO: \mathfrak{C}(M) \times (C^\infty(M,\R))^N_{\text{\rm reg}} \to \R^N.
$$
and the 1-jet bundle $J^1 \R^N$. Denote by
$$
\mathfrak{Leg}(J^1\R^N)
$$
the set of Legendrian (immersed) submanifolds. Then we have the natural map
$$
\mathfrak{TdR}_N: \mathfrak{C}(M) \times (C^\infty(M,\R))^N_{\text{\rm reg}} \to \mathfrak{Leg}(J^1\R^N)
$$
given by
$$
(\lambda,\CF) \mapsto R_{\CS_{\lambda_0};\CF}
$$
which we call the \emph{thermodynamic reduction of $N$-Hamiltonian systems}.
\end{defn}
We would like to regard $R_{\CS_{\lambda_0};\CF}$ as a \emph{thermodynamic invariant}
of the multi-Hamiltonian system.

The followings are natural questions to ask which we will study
elsewhere.
\begin{prob} 
\begin{enumerate}
\item Describe the contact thermodynamic equilibrium $R_{\CS_{\lambda_0};\CF}$ in terms of 
the dynamics of the contact multi-Hamiltonian system $\CF = \{F_i\}_{i=1}^N$.
\item Study the phase transition or the wall-crossing of the
thermodynamic invariant map $\mathfrak{TdR}$ or of its derived.
\item To what extent does $R_{\CS_{\lambda_0};\CF}$ encode  dynamical information of 
$\CF$?
\end{enumerate}
\end{prob}

\section{Transversality of the reduced entropy function as a generating function}
\label{sec:transversality}

This section will be occupied by the transversality condition
\eqref{eq:transversality-dvCSCF}.
We observe that all $\CS_{\lambda_0}$, $\CO_\CF$  factorize to the
corresponding maps defined on $\CD^+(M)$ via the map $\vol$. We already have
$$
\CS_{\lambda_0} = D_{\text{\rm KL}}^{\mu_{\lambda_0}}\circ \vol.
$$
We also write
$$
\CO_\CF = \CO_\CF^{\CD^+} \circ \vol.
$$
For $\CS_{\lambda_0}$, a standard calculation leads to
\be\label{eq:dSSlambda0}
d\CS_{\lambda_0}(\lambda) (\alpha) = \int_M (\log f_{\lambda;\lambda_0}+1)
d\vol(\lambda)(\alpha).
\ee

\subsection{On the small phase space}

We start with the small phase space $\mathfrak C^+(M,\xi)$
and the dimension formula of $R_{\CS_{\lambda_0};\FF}$.

On the small phase space $\mathfrak C^+(M,\xi)$, we have the first variation of the form
$$
\delta_{\alpha} \mu_\lambda = (n+1)h_i \, \mu_\lambda
$$
if we write $\alpha = \delta \lambda = h_i \lambda$.
Suppose  $\lambda \in \Sigma_{\CF;\CS_{\lambda_0}^{\text{\rm sm}}}$ i.e.,  $\lambda$ satisfies
\be\label{eq:dCO=0}
d^v\CS_{\lambda_0}^{\text{\rm sm}}(\lambda) = 0.
\ee
Regarding $d^v\CS_{\lambda_0}^{\text{\rm sm}}=: \aleph$ as a section of
the vector bundle  
$$
\pi: (VT)^*\mathfrak C^+(M,\xi)  \to \mathfrak C^+(M,\xi)
$$
whose fiber is given by 
$$
(VT)_\lambda^* \mathfrak C^+(M,\xi): =(VT_\lambda\mathfrak C^+(M,\xi))^* \cong 
(HT_\lambda\mathfrak C^+(M,\xi))^\perp.
$$
We also decompose its derivative  $D\aleph$  into
$$
D\aleph = D_h\aleph + D_v\aleph
$$
according to the decomposition of the \emph{domain} of the operator $D\aleph$
$$
T_\lambda\mathfrak C^+(M,\xi) = HT_\lambda\mathfrak C^+(M,\xi) \oplus VT_\lambda\mathfrak C^+(M,\xi), 
$$
i.e., $D_h\aleph$ and  $D_v\aleph$ are the \emph{restrictions} of $D\aleph$ to
the first and the second summands respectively. (We attract the readers' attention
to the location of the indices `v' and `h', which are subindices differently form the supindices 
of the differential $d^v\CS_{\lambda_0}$ and $d^h\CS_{\lambda_0}$ which are the 
\emph{projections}.)

We mention 
\be\label{eq:HV-isomorphism}
HT_\lambda \mathfrak C^+(M,\xi) \cong T_{\CO_\CF(\lambda)} B_\CF, \quad 
VT_\lambda \mathfrak C^+(M,\xi) \cong T_\lambda \left(\CO_\CF^{-1}\left(\CO_\CF(\lambda)\right)\right).
\ee
By definition, we have the covariant derivative
$$
\nabla_{\alpha_2}\aleph \Big|_\lambda(\alpha_1): = (D^v_{\alpha_2}\aleph)(\alpha_1).
$$
Then we can express the covariant Hessian as
$$
\text{\rm VHess}\, \CS_{\lambda_0}^{\text{\rm sm}}(\alpha_1, \alpha_2) 
: = \nabla_{\alpha_2}(d\CS_{\lambda_0}^{\text{\rm sm}})\Big|_\lambda(\alpha_1) = \nabla_{\alpha_2}\aleph \Big|_\lambda(\alpha_1), 
\quad \alpha_i = h_i \lambda
$$
as a symmetric bilinear form on $VT_\lambda \mathfrak C^+(M,\xi)$ at the zero point $\lambda$ 
of the section  $\aleph = d^v \CS_{\lambda_0}^{\text{\rm sm}}$.

\begin{lem}\label{lem:Hessian-small} Write $\alpha_2 = h_2 \lambda$ and $\alpha_1 =h_1 \lambda$. Then
$$
\text{\rm Hess}\, \CS_{\lambda_0}^{\text{\rm sm}}(\alpha_1, \alpha_2) 
= \int_M \left((n+1)(2n+1) + n (n+1)\log f_{\lambda; \lambda_0}\right) h_1 h_2\, d\mu_\lambda d\mu_\lambda
$$
\end{lem}
\begin{proof} 
Let $\alpha_1, \, \alpha_2$ be two vertical variations at $\lambda$ given by 
\beastar
\alpha_1 & = & \dfrac{\del}{\del s} \Big|_{(s,t) = (0,0)} \lambda_{s,t}\\
\alpha_2 & = & \dfrac{\del}{\del t} \Big|_{(s,t) = (0,0)} \lambda_{s,t}
\eeastar
which are in $T_{\lambda}(\EC^+(M,\xi))$. We can write $\mu_{\lambda_{(s,t)}} = f_{(s,t)} \mu_{\lambda_0}$
with $f_{(s,t)} = f_{\lambda_{s,t};\lambda_0}$. 

We note $f_{(0,0)} = f$ with $\lambda = f\, \lambda_0$ and put
$$
h_1 = \frac{\del f_{(s,t)}}{\del s}\Big|_{(s,t) = (0,0)}, \quad h_2 =  \frac{\del f_{(s,t)}}{\del t}\Big|_{(s,t) = (0,0)}.
$$
Then we have  $\alpha_i = h_i \lambda$ at $\lambda \in \Sigma_{\CS_{\lambda_0};\CF}^{\text{\rm sm}}$, and
compute the formula for the (full) second variation at $\lambda$
\be\label{eq:2nd-variation-CS}
\frac{\del^2}{\del s \del t} \Big|_{(s,t) = (0,0)}\CS_{\lambda_0}^{\text{\rm sm}}(\lambda_{s,t})
= \frac{\del^2}{\del s \del t}\Big|_{(s,t) = (0,0)} \int_M  f_{(s,t)}
 \log f_{(s,t)}\, d\mu_{\lambda_0}
\ee
and
\beastar
&{}& \frac{\del^2}{\del s \del t}\Big|_{(s,t) = (0,0)} f_{(s,t)}
 \log f_{(s,t)} \\
 & =  &(1 + \log f_{(0,0)}) \frac{\del ^2 f_{(s,t)}}{\del s\del t}\Big|_{(s,t)=(0,0)} 
 + \frac{1}{f_{(0,0)}}\frac{\del f_{(s,t)}}{\del s} \frac{\del f_{(s,t)}}{\del t} \Big|_{(s,t)=(0,0)}.
 \eeastar

 After integration,  we derive
 \beastar
 \frac{\del^2}{\del s \del t} \Big|_{(s,t) = (0,0)}\CS_{\lambda_0}^{\text{\rm sm}}(\lambda_{s,t})
 & = & \int_M (1 + \log f_{(0,0)} ) \frac{\del ^2 f_{(s,t)} }{\del s\del t}\Big|_{(s,t)=(0,0)}  
 \, d\mu_{\lambda_0} \\
&{}& \quad   + \int_M \frac{1}{f_{(0,0)}}\frac{\del f_{(s,t)}}{\del s} \frac{\del f_{(s,t)}}{\del t} \Big|_{(s,t)=(0,0)}\, d\mu_{\lambda_0}\\
& = & \int_M (1 + \log f_{(0,0)} ) \frac{\del ^2 f_{(s,t)} }{\del s\del t}\Big|_{(s,t)=(0,0)}  \, d\mu_{\lambda_0} \\
&{}& \quad +
\int_M k_1 k_2  \, d\mu_{\lambda}.
 \eeastar
Recall $k_i  : = \delta_{\alpha_i} \lambda = (n+1) h_i$ for $\alpha_i = h_i \, \lambda$.
Then we obtain
$$
\frac{\del ^2 f_{(s,t)}}{\del s\del t} \Big|_{(s,t)=(0,0)}=(n+1)n f_{\lambda;\lambda_0}\, h_1 h_2.
$$ 
Combining the above, we have proved
$$
\nabla_{\alpha_2}(d\CS_{\lambda_0}^{\text{\rm sm}})\Big|_\lambda(\alpha_1)
= \int_M \left((n+1)(2n+1) + n (n+1)\log f_{\lambda; \lambda_0}\right) h_1 h_2\, d\mu_\lambda.
$$
This finishes the proof.
\end{proof}

Note that the map $D_v\aleph(\lambda)$ is nothing but the linear map
associated to the symmetric bilinear form $\text{\rm VHess}_\lambda \CS_{\lambda_0}$ and so it is a 
symmetric operator in the $L^2$ sense.
To make our discussion completely rigorous, we need to take 
a suitable completion of off-shell space $\mathfrak C(M,\xi)$.
Since this functional analytic discussion is not relevant to our main purpose, we will not elaborate it here.

\begin{prop}\label{prop:SSsm-transversality} Suppose $\CF$ satisfies 
the regularity criterion given in Proposition \ref{prop:regularity-small}. 
Suppose the germ of $\{1, F_1, \cdots, F_n\}_{i=1}^N$ at any point of M is linearly independent.
Then the map
$$
D\aleph(\lambda): T_\lambda\mathfrak C^+(M,\xi) \to VT_\lambda^* \mathfrak C^+(M,\xi) \cong 
\left(HT_\lambda \mathfrak C^+(M,\xi)\right)^\perp 
$$
has dense image.
\end{prop} 
\begin{proof} It is enough to show that the $L^2$-cokernel of $D\aleph(\lambda)$, 
i.e., the kernel of formal adjoint  (or $L^2$-adjoint) $D^\dagger \aleph(\lambda)$ vanishes. 
 
Suppose a vertical vector field
$$
Y \in VT_\lambda\mathfrak C^+(M,\xi) \cong T_\lambda \left(\CO_\CF^{-1}(\CO_\CF(\lambda) \right) \subset T_\lambda \mathfrak C^+(M,\xi)
$$
satisfies
$$
\langle\langle X, D^\dagger \aleph(\lambda) (Y)\rangle 
= \langle \langle D\aleph(\lambda) (X), Y \rangle \rangle = 0
$$
for all $X \in T_\lambda\mathfrak C^+(M,\xi)$. Here $\langle \langle \cdot, \cdot \rangle \rangle$
stands for the $L^2$-paring for the measure $\mu_\lambda$. We would like to conclude $Y = 0$.

By writing $X = X^\pi + X^\perp$, this equation becomes
$$
\langle \langle D\aleph(\lambda)( X^\pi), Y \rangle \rangle
 + \langle \langle D\aleph(\lambda) (X^\perp), Y \rangle \rangle = 0.
$$
In the case of our current interest, it is enough to consider $X$ with $X^\pi = 0$. 
Then the equation becomes
\beastar
0 & = & \langle \langle D\aleph(\lambda) (X^\perp), Y \rangle \rangle
= \langle \langle D_v \aleph(\lambda)( X^\perp), Y \rangle \rangle \\
& = & \text{\rm VHess}\, \CS_{\lambda_0}(\lambda) (X^\perp, Y).
\eeastar
If we set $X^\perp = h_1 \lambda, \, Y = h_2 \lambda \in VT_\lambda \mathfrak C^+(M,\xi)$, then
$$
\text{\rm VHess} \CS_{\lambda_0}(\lambda) (X^\perp,Y) 
= \text{\rm Hess}\CS_{\lambda_0}(\lambda) (X^\perp,Y)
$$
and hence  Lemma \ref{lem:Hessian-small} gives rise to
\beastar
0 & = &  \int_M (n+1)(2n+1) h_1 h_2\, d\mu_\lambda + 
n \int_M (n+1)(1+ \log f_{\lambda; \lambda_0}) h_1 h_2\, d\mu_\lambda\\
& = & \int_M (n+1)(2n+1) h_1 h_2\, d\mu_\lambda + n d\CS_{\lambda_0}^{\text{\rm sm}}(\lambda)( (h_1h_2) \, \lambda)
\eeastar
for all function $h_1$. Recalling
$$
d\CS_{\lambda_0}^{\text{sm}}(\lambda) = \sum_{i=1}^N p_i d\CO_{F_i}^{\text{sm}}(\lambda), \, \text{ for some } p_i \in \R,
$$
we show
$$
d\CS_{\lambda_0}^{\text{sm}}(\lambda) (h_1h_2) =  \sum_{i=1}^N p_i d\CO_{F_i}^{\text{sm}}(\lambda)(h_1h_2)
$$
we have concluded
$$
\int_M \left( (n+1)(2n+1)+ n (n+1) \sum_{i=1}^N  p_i F_i \right) h_1h_2  \, d\mu_\lambda = 0
$$
for all $h_1$.  This implies
$$
\left( (n+1)(2n+1) + n(n+1) \sum_{i=1}^N p_i F_i\right) h_2 = 0.
$$
Let $x \in M$ be any point. By the hypothesis, there is an open neighborhood $U$ of $x$ such that
$\{1, \, F_1, \, F_2, \cdots, F_N\}|_U$ is linearly independent and so the function
$$
\left( (n+1)(2n+1)+ n(n+1) \sum_{i=1}^N  p_i F_i \right)\Big|_U
$$
is not a zero function. Therefore $h_2(y) = 0$ for some
point $y \in U$. Since this holds for any neighborhood of $x$ and $h_2$ is continuous, $h_2(x) = 0$.
Since this holds for all $x \in M$, we conclude $h_2= 0$ and hence $Y = 0$.
This finishes the proof.
\end{proof}

An immediate corollary is the following.
\begin{cor}\label{cor:dimension} Assume the map
$D_v\aleph(\lambda): VT_{\aleph(\lambda)} \mathfrak C(M,\xi) \to  VT_{\aleph(\lambda)}^* \mathfrak C(M,\xi)$
is an isomorphism. Then under the same hypotheses as in Proposition
\ref{prop:SSsm-transversality}, the following holds:
\begin{enumerate}
\item $\CS_{\lambda_0}^{\text{\rm sm}}$ is relative Morse with respect to the fibration 
$\CO_\CF^{\text{\rm sm}}: \mathfrak{C}(M,\xi) \to B_\CF$ in the sense of Definition \ref{defn:relative-Morse}.
\item 
$\Sigma_{\CS_{\lambda_0}^\text{\rm sm}}$ is an $N$ dimensional smooth submanifold.
\end{enumerate}
\end{cor}
\begin{proof} By Proposition \ref{prop:SSsm-transversality}, smoothness follows.
Furthermore Hypothesis (1) also implies that $\CS_{\lambda_0}^{\text{\rm sm}}$ is relative Morse.

It is enough to compute its dimension. We recall the decomposition 
$$
T_\lambda\mathfrak C^+(M,\xi) = HT_\lambda\mathfrak C^+(M,\xi) \oplus VT_\lambda \mathfrak C^+(M,\xi),
$$
and that $HT_\lambda\mathfrak \mathfrak C^+(M,\R)$ is of the same dimension as $B_\CF = N$.
We also recall that $\aleph$ is a section
of $VT^*\mathfrak C^+(M,\xi) \to \mathfrak C^+(M,\xi)$ so that $\pi \circ \aleph = \id$ and hence
$$
d\CO_\CF(\lambda)|_{HT} : HT_\lambda \mathfrak C^+(M,\xi) \to T_{\CO_\CF(\lambda)} B_\CF
$$
 is an isomorphism. See the following commutative diagram:
 $$
 \xymatrix{T_\lambda \mathfrak C^+(M,\xi)\ar[d] _{\pi_h} \ar[dr]^{D\aleph(\lambda)} \ar[r]^{\pi_v}
 & VT_\lambda \mathfrak C^+(M,\xi)\ar[d]^{D_v\aleph(\lambda)} \\
   HT_\lambda \mathfrak C^+(M,\xi) \ar[d]_{d_\lambda \CO_\CF|_{HT}}  \ar[r] 
   \ar[dr]^{\hookrightarrow} 
 \ar[r]^<<<<{D_h\aleph(\lambda)}  &T_{\aleph(\lambda)}\left( (VT)^*\mathfrak C^+(M,\xi)\right) \ar[d]^{d\pi} & \\
 T_{\CO_\CF(\lambda)} B_\CF   &  T_{\lambda}\mathfrak C^+(M,\xi)  \ar[l]_{d_\lambda \CO_\CF}  
 }
 $$
 where we use the identification 
 $$
T_\lambda \mathfrak C^+(M,\xi)  \oplus (VT)^*_{\lambda}\mathfrak C^+(M,\xi)
\cong T_{\aleph(\lambda)}\left( (VT)^*\mathfrak C^+(M,\xi)\right).
$$
Obviously we have $d\pi \circ D\aleph(\lambda) = \id$ since $\aleph$ is a section.
Combining this with the sujectivity of $D_v \aleph (\lambda)$ implies $D\aleph(\lambda)$ is 
surjective.  We consider the following commutative diagram
$$
\xymatrix{ 
0 \ar[r]  & T_\lambda \Sigma_{\CS_{\lambda_0};\CF}  \ar[r] \ar[d]_{d_\lambda \CO_\CF|_\Sigma}
& T_\lambda \mathfrak C(M,\xi) \ar[d]_{=} \ar[r]^<<<<{D\aleph(\lambda)} & VT_{\aleph(\lambda)}^* \mathfrak C(M,\xi) \ar[d]_{D_v^*\aleph(\lambda)}
 \ar[r] & 0\\
0 \ar[r]  & T_{\CO_\CF(\lambda)} B_\CF \ar[r]  &T_\lambda \mathfrak C(M,\xi) \ar[r] &  VT_{\aleph(\lambda)}\mathfrak C(M,\xi)\ar[r] 
 & 0
 }
$$
where the first map of the second row is the horizontal lifting and the second map is the
composition
$$
D_v^*\aleph(\lambda) \circ D\aleph(\lambda).
$$
Exactness of the first row is obvious by the definitions of $\Sigma_{\CS_{\lambda_0};\CF}$ and of the map $\aleph$. 
It follows from nondegeneracy of $\text{\rm VHess} \CS_{\lambda_0}(\lambda)$ that the second row is also
exact, whose proof we omit.
Therefore, under the given hypothesis,  Five Lemma implies that the map $d_\lambda \CO_\CF|_\Sigma$ is an isomorphism.
Combining these, we conclude that  $\dim \Sigma_{\CS_{\lambda_0}^\text{\rm sm}} = N$.
\end{proof}

\begin{rem} Recall that $D\aleph(\lambda)$ is a symmetric operator in the $L^2$ space.
We take a self-adjoint extension thereof.
(We refer readers to \cite{rudin:functional}, for example, for the detailed study of this extension.)
Therefore, once we have the denseness proved in Proposition \ref{prop:SSsm-transversality}, 
it will automatically imply Hypothesis (1) required in the above corollary,
after taking a suitable completion mentioned right before Proposition \ref{prop:SSsm-transversality}. 
This is because the extension will also have zero cokernel, which in turn implies zero kernel by the self-adjointness.
\end{rem}

\subsection{On the big phase space}

Recall $\CS_{\lambda_0} = D_{\text{\rm KL}}^{\mu_{\lambda_0}}\circ  \vol$. Therefore
we have $d\CS_{\lambda_0} = d D_{\text{\rm KL}}^{\mu_{\lambda_0}}\circ  d\vol$ and
the second derivative
\be\label{eq:d2CS}
d^2\CS_{\lambda_0} = d^2 D_{\text{\rm KL}}^{\mu_{\lambda_0}} ( d\vol, d\vol) + dD_{\text{\rm KL}}^{\mu_{\lambda_0}}\circ (d^2 \vol)
\ee
and similar formulae apply to $\CO_\CF$ and $V$. The following lemma is standard.

\begin{lem}\label{lem:hessian} Consider any
 function $T^{\mathfrak C^+}: \mathfrak{C^+}(M) \to \R$ of the form
$T^{\mathfrak C^+} = T^{\CD^+} \circ \vol$. Let $\lambda$ be a critical point of $T^{\mathfrak C}$ and
$\alpha_i$ $i=1, \, 2$ be two variations of $\lambda$. 
Decompose $\alpha_i = Y_i^\pi \intprod d\lambda + h_i \lambda$ for $i =1, \, 2$.
Then we have
$$
d^2 T^{\mathfrak C^+}(\lambda) (\alpha_1, \alpha_2) 
= d^2 T^{\CD^+}(\mu_\lambda)(\delta_{\alpha_1}\mu_\lambda,\delta_{\alpha_2}\mu_\lambda)
$$
 for $i=1, \, 2$.
\end{lem}
Having the above general second derivative formula in our mind and the lemma, we compute
\beastar
d^2\vol(\lambda)(\alpha_1, \alpha_2)&=&n d \alpha_2  \wedge \alpha_1 \wedge (d\lambda)^{n-1}+ n d \alpha_1 \wedge \alpha_2 \wedge (d\lambda)^{n-1}\\
&+&
n(n-1) \lambda \wedge (d\alpha_1) \wedge (d\lambda)^{n-2} \wedge (d\alpha_2).
\eeastar
We put  $\delta_{\alpha_i} \mu_\lambda = k_i \, \mu_\lambda$
for some function $k_i$ respectively. Then decomposing $\alpha_i$ into 
$\alpha_i = Y^\pi_i \intprod d\lambda + h_i \lambda$, we derive
\be\label{eq:ki}
k_i = (n+1) h_i + \nabla \cdot Y^\pi_i 
\ee
on the big phase space. (See \eqref{eq:dvol} in Proposition \ref{prop:dvol}.)
Then  we obtain
\bea\label{eq:d2vol-lambda}
d^2\vol(\lambda)(\alpha_1, \alpha_2)&=& n(n+1) h_1 h_2 \, \mu_\lambda + nh_1 \nabla\cdot Y_2^\pi\, \mu_\lambda 
+ nh_2 \nabla\cdot Y_1^\pi\, \mu_\lambda \nonumber\\
&{}& + (Y_1^\pi[h_2] + Y_2^\pi[h_1])\, \mu_\lambda \nonumber\\
&{}& + n\CL_{Y_1^\pi}\lambda \wedge \CL_{Y_2^\pi}d\lambda \wedge (d\lambda)^{n-1} \nonumber\\
&{}& + n\CL_{Y_2^\pi} \lambda \wedge \CL_{Y_1^\pi}d\lambda  \wedge (d\lambda)^{n-1} \nonumber\\
&{}& + n(n-1)\lambda \wedge (\CL_{Y_1^\pi}d\lambda) \wedge (\CL_{Y_2^\pi}d\lambda) \wedge (d\lambda)^{n-2}
\eea
and
\be
\text{Hess }\CS_{\lambda_0}(\lambda)(\alpha_1, \alpha_2)=\int_{M} k_1 k_2 d\mu_{\lambda}+\int_{M} (1 + \log f_{\lambda; \lambda_0})d^2\vol(\lambda)(\alpha_1, \alpha_2).
\ee

The following explicit formula for the Hessian of $\CS_{\lambda_0}(\lambda)$ will be derived
by taking one more derivative of \eqref{eq:dSSlambda0}. 

We postpone its proof until Appendix \ref{sec:3-terms}.

\begin{prop}\label{prop:Hessian-big}  Let $\alpha_i = Y_i^\pi \intprod d\lambda + h_i\, \lambda$ for $i=1, \, 2$,
$\mu_{\lambda} = f \mu_{\lambda_0}$ with $f = f_{\lambda;\lambda_0}$, and $k_i = (n+1) h_i + \nabla \cdot Y^\pi_i$ for $i=1,2.$ 
Denote by $X_g$ the Hamiltonian vector field of $g: = 1 + \log f$.
Then
\bea\label{eq:Hessian-big} 
&{}& \text{\rm Hess} \CS_{\lambda_0}(\lambda)(\alpha_1,\alpha_2) \nonumber \\
 & = &\int_M k_1 k_2\,d\mu_{\lambda} \nonumber\\
&{}&+ \int_M n (n+1) h_1 h_2 (1+\log f) \,  d\mu_\lambda \nonumber \\
&{}&-(n+1) \int_M g(Y_2^\pi[h_1]+Y_1^\pi[h_2])\,d\mu_{\lambda} \nonumber\\
&{}&-n \int_M (h_1 Y_2^\pi[g]+h_2 Y_1^\pi[g])\,d\mu_{\lambda} \nonumber\\
 &{}&+\int_M (d\lambda(Y_2^\pi, [Y_1^\pi, X_g^\pi])+2 \, d\lambda(Y_1^\pi, [Y_2^\pi, X_g^\pi]))\,d\mu_{\lambda}  \nonumber \\
 &{} & \int_{M} (\lambda([Y_2^\pi,[Y_1^\pi, X_g^\pi]])+\lambda([Y_1^\pi,[Y_2^\pi, X_g^\pi]]))\,d\mu_{\lambda} \nonumber\\
 &{}&-(n+1)\int_{M} g\lambda([Y_1^\pi, [Y_2^\pi, R_{\lambda}]])\,d\mu_{\lambda}
\eea
\end{prop}
This is supposed to be symmetric with respect to the swapping of $1$ and $2$. In particular, the
sum of  the second and third lines is supposed to be symmetric, which is, however,
is not apparent at all. One might want
to symmetrize the expressions in the two lines to manifest the symmetry.

We observe that the first and second summand of \eqref{eq:Hessian-big} is precisely the same the Hessian
in Lemma \ref{lem:Hessian-small}
for the small phase space as expected, since it corresponds to $Y_i^\pi = 0$.

Using this observation and the explicit formula, we obtain the following
transversality criterion on the big phase space similarly as 
in the small phase space.  We leave the details of the proof to the interested 
readers except mentioning that  we have
$$
X_{-g}= X_{- (1+\log f_{\lambda;\lambda_0})} = X_{-1} = R_{\lambda_0}
$$
and that the bilinear form appearing in the middle vanishes when $f_{\lambda; \lambda_0}=1$.

\begin{prop} \label{prop:SSbig-transversality} Assume that the system $\CF = \{ F_1, \cdots, F_N\}$ satisfies 
that the germ of $\{X_{F_1}, \cdots, X_{F_N}\}$ at any point $M$ is linearly independent
in addition to the hypothesis of Proposition \ref{prop:SSsm-transversality}.
Then $\CS_{\lambda_0}$ is relative Morse with respect to the fibration $\CO_\CF: \mathfrak C(M) \to B_\CF$.
\end{prop}

\section{Description of thermodynamic equilibrium distribution}
\label{sec:description}

We now fix the contact form $\lambda_0 \in \mathfrak{C}^+(M)$ 
and consider the system with $N$-observables,  $F_1,\dots,F_N$.
The discussion of this section applies both to the big and the small phase spaces respectively.
For the simplicity of notation, we just write $\CS_{\lambda_0}$ for both cases
in this section.

Then, we want to find a $\lambda \in \EC^+(M)$  that extremizes the
relative entropy functional
$$
\CS_{\lambda_0} (\lambda)= \int_M  \log \left(\frac{d\mu_\lambda}{d \mu_{\lambda_0}}\right) d \mu_\lambda.
$$
under the constraints
\begin{equation}\label{eq:pi=intFi}
 q_i  = \int_M F_i d\mu_\lambda
\end{equation}
for given explicit observations $q=(q_1,\cdots,q_N)$.

In particular at any extremum point $\lambda \in \EC^+(M)$, $dS_{\lambda_0}(\lambda)$ must vanish 
on $T_{\lambda}(\CO_{\CF}^{-1}(q))
=\text{ker }d\CO_{\CF}(\lambda)$ with $q =\CO_{\CF}(\lambda)$.
(\emph{We would like to emphasize that the restriction 
$$
\CO_\CF|_{\Sigma_{\CS_{\lambda_0;\CF}}}: \Sigma_{\CS_{\lambda_0;\CF}} \to B_\CF
$$
of the observation map $\CO_\CF$  is still not necessarily one-to-one.})
The vertical criticality is equivalent to 
\be\label{eq:conormal-condition}
d\CS_{\lambda_0}(\lambda) = \sum_{i=1}^N p_i\, d\CO_{F_i}(\lambda),
\ee
for some $(p_1, \cdots, p_N) \in \R^{N}$.

Now we can re-write \eqref{eq:conormal-condition} into
\be\label{eq:exact-differential}
d\left(\CS_{\lambda_0}-\sum_{i=1}^N p_i \CO_{F_i}\right)(\lambda) (\alpha) = 0
\ee
for all $\alpha \in \Omega^1(M)$. Once we have these formulae, we easily derive
\be \label{eq:dSlambda0-alpha}
dS_{\lambda_0}(\lambda)(\alpha) = 
\int_M  (\log f_{\lambda;\lambda_0} + 1)\,  \delta_\alpha \mu_\lambda 
\ee
where we recall $\delta_\alpha \mu_\lambda =  \delta f_{\lambda;\lambda_0} \, d\mu_{\lambda_0}.$
We also have
\be\label{eq:dCOFi-alpha}
d\CO_{F_i}(\lambda)(\alpha)=\int_M F_i \, \delta_\alpha \mu_\lambda, \, i=1, \cdots, N,
\ee
for all $\alpha \in \Omega^1(M)$.

These being said, \eqref{eq:exact-differential} can be written as
\beastar
0  &= & 1 + \log f_{\lambda;\lambda_0} - \sum_{i=1}^N p_i F_i.
 \eeastar
and hence
\be\label{eq:fD}
f_{\lambda;\lambda_0} = e^{(\sum_{i=1}^N p_i F_i)-1}
\ee
Now at any point $\lambda$ satisfying $d^v\CS_{\lambda_0}(\lambda) = 0$,
the regularity stated in Proposition \ref{prop:regularity}
implies that  the set $\{d\CO_{F_i}(\lambda)\}$ is linearly 
independent so that $\lambda$ admits  a smooth branch of Lagrange multipliers
$(p_1, \cdots, p_N)$  satisfying \eqref{eq:conormal-condition}. This provides a local coordinate 
description of
the contact thermodynamic equilibrium $R_{\CS_{\lambda_0};\CF}$. 

\begin{rem}\label{rem:multi-valued} Recall that each branch is locally given by the intersection
$$
\CO_\CF^{-1}(q) \cap \Sigma_{\CS_{\lambda_0};\CF}.
$$
The above Lagrange multiplier provides the local description of the image of 
$$
\iota_{\CS_{\lambda_0};\CF} : \Sigma_{\CS_{\lambda_0};\CF} \to J^1B_\CF,
$$
near the given image point 
$$
(q,p, z) =  \iota_{\CS_{\lambda_0};\CF} (\lambda).
$$
\emph{There could be more than one element $\lambda \in \CO_\CF^{-1}(p) \cap \Sigma_{\CS_{\lambda_0};\CF}$
that have the same image in $B_\CF$ under the observation map, and also
in the image of the map $\iota_{\CS_{\lambda_0};\CF} $ in $J^1B_\CF$.}
(See Appendix \ref{sec:Legendrian-property} for the precise description.)
This is the reason why this
coordinate discussion works only locally and it is also the source of \emph{the failure of 
the $C^0$-holonomicity} discussed in \cite{lim-oh} in their discussion of the Maxwell's 
equal area law. 
As mentioned therein, our discussion of the reduced entropy function 
as a generating function of the thermodynamic equilibria provides a global functorial
discussion thereof. We would also like to mention that this is also the reason 
why $R_{\CS_{\lambda_0};\CF}$ is a priori
only an (immersed) submanifold, not necessarily embedded. 
(See Appendix \ref{sec:Legendrian-property}.)
\end{rem}

If we fix $\lambda_0 \in \EC_1(M)$ and $\CF=\{1, F_1, \cdots, F_N\}$, then (\ref{eq:conormal-condition}) can be changed to 
\bea
d\CS_{\lambda_0}(\lambda)=(1-w)dV(\lambda)+\sum_{i=1}^N p_i d\CO_{F_i}(\lambda)
\eea
for some $(w, p_1, \cdots, p_N )$, where $V$ is the mass function.
Then, if we introduce an additional condition $V(\lambda) = \int_M\mu_\lambda = 1$ (the normalization condition), we obtain
$$
1 = \int_M  d\mu_{\lambda} = \int_M e^{-w + \sum_{i=1}^N p_i F_i} \, d\mu_{\lambda_0}
$$
and the expression of $w$ in terms of other Lagrange multipliers
\begin{equation}\label{eq:ew}
e^{w} = \int_M e^{\sum_{i=1}^N p_i F_i}\, d\mu_{\lambda_0}.
\end{equation}
By taking the logarithm thereof, we get
\begin{equation}\label{eq:w}
w  = \log \left( \int_M e^{\sum_{i=1}^N  p_i F_i}\, d\mu_{\lambda_0}\right).
\end{equation}

Therefore we have proved
\begin{prop} At any extreme point $\rho_{(w,p);\lambda_0}$ of the entropy under the constraint
\eqref{eq:pi=intFi}, there exists a Lagrange multiplier
$(w, p_1, \ldots, p_N)$ such that
$\mu_\lambda$ has the form
$$
\rho_{(w,p);\lambda_0} = e^{-w + \sum_{i=1}^N p_i F_i}\, \mu_{\lambda_0}.
$$
For the measure, 
$$
\rho_{p;\lambda_0} = e^{-w(p) + \sum_{i=1}^N p_i F_i}\, \mu_{\lambda_0
}$$
where $1=\int_M \mu_{\lambda}$, the followings also hold:
\begin{enumerate}
\item
The $\rho_{ p; \lambda_0}$ has the form
\be\label{eq:rho-formula}
\rho_{p; \lambda_0} = \left(\frac{e^{\sum_{i=1}^N p_i F_i}}{\int_M
e^{\sum_{i=1}^N p_i F_i}\,  d
\mu_{\lambda_0}} \right) \, \mu_{\lambda_0}.
\ee
\item The observations associated to the variable $(p_1, \cdots, p_N)$
have their values
$$
q_i(p):=\int_M F_i d\mu_{\lambda} = \frac{\partial w}{\partial{p_i}}, 
$$
\item The (relative) entropy is expressed as
$$
\CS_{\lambda_0} (p)= \int_{M} \left(-w(p) +\sum_{i=1}^N p_i F_i\right)\, e^{-w(p) +\sum_i p_i F_i}\, d\mu_{\lambda_0}
$$
on the thermodynamic equilibrium $R_{\CS_{\lambda_0};\CF}$ given in \eqref{eq:RCF}.
\end{enumerate}
\end{prop}

A natural question to ask is the following.

\begin{ques} 
Obviously $\rho_{p;\lambda_0}$ is an equilibrium state when $p = 0$.
When will the equilibrium state $\rho_{p;\lambda_0}$ be of the form $\psi_*\mu_{\lambda_0}$
for some diffeomorphism $\psi$ \emph{for $p \neq 0$?} If not for diffeomorphism $\psi$, will it be possible if we reduce the regularity of $\psi$ to find one such?
\end{ques}

\appendix

\section{Proof of Proposition \ref{prop:Hessian-big}}
\label{sec:3-terms}

In this appendix, we prove Proposition \ref{prop:Hessian-big}. In our calculation, the following
 general lemma will be frequently used.

\begin{lem}\label{lem:useful1} For any vector field $Y^\pi$ tangent to $\xi$, we have
$$
(\CL_{Y^\pi} \lambda) \wedge (d\lambda)^n = 0.
$$
\end{lem}
\begin{proof}
By Cartan's magic formula and $\lambda(Y^\pi) = 0$, we derive
$$
(\CL_{Y^\pi} \lambda) \wedge (d\lambda)^n = (Y^\pi \intprod d\lambda) \wedge (d\lambda)^n.
$$
Note that $(Y^\pi \intprod d\lambda)(R_\lambda) = 0$. Therefore we conclude 
The one form $Y^\pi \intprod d\lambda$ should be a linear combination of  
$\xi^* = (R_\lambda)^\perp$ under the decompostion $TM = \xi \oplus \R \langle R_\lambda \rangle$.
Therefore the lemma follows because the $2n$-form $(d\lambda)^n$ already exhausts all
such one-forms in its wedge product expansion, say in Darboux frame.
\end{proof}
%
%

The following lemma  will be also needed.

\begin{lem} \label{lem:2-terms} We have
\bea
 n\CL_{Y_1^\pi}\lambda \wedge \CL_{Y_2^\pi}d\lambda \wedge (d\lambda)^{n-1} & = & 
- \CL_{Y_2^\pi}\CL_{Y_1^\pi}\lambda \wedge (d\lambda)^n  \label{eq:first-term}\\
  n \CL_{Y_1^\pi}\lambda \wedge \CL_{Y_2^\pi}d\lambda \wedge (d\lambda)^{n-1} & = & 
- \CL_{Y_1^\pi}\CL_{Y_2^\pi}\lambda \wedge (d\lambda)^n
  \label{eq:second-term}
\eea
\end{lem}
\begin{proof} Utilizing Lemma \ref{lem:useful1}, we derive
\beastar
 n\CL_{Y_1^\pi}\lambda \wedge \CL_{Y_2^\pi}d\lambda \wedge (d\lambda)^{n-1}
 & = & \CL_{Y_1^\pi}\lambda \wedge \CL_{Y_2^\pi} (d\lambda)^n \\
 & = & -\CL_{Y_2^\pi}\CL_{Y_1^\pi}\lambda \wedge (d\lambda)^n
 \eeastar
 which proves the first equality. The second is the same by swapping 1 and 2.
\end{proof}

Using the latter lemma, we first transform \eqref{eq:d2vol-lambda} into 
\beastar
 d^2\vol(\lambda)(\alpha_1, \alpha_2) 
& = & n(n+1) h_1 h_2 \, \mu_\lambda + nh_1 \nabla\cdot Y_2^\pi\, \mu_\lambda 
+ nh_2 \nabla\cdot Y_1^\pi\, \mu_\lambda\\
&{}& + (Y_1^\pi[h_2] + Y_2^\pi[h_1])\, \mu_\lambda \\
&{}&  - \CL_{Y_2^\pi}\CL_{Y_1^\pi}\lambda \wedge (d\lambda)^n 
- \CL_{Y_1^\pi}\CL_{Y_2^\pi}\lambda \wedge (d\lambda)^n \\
&{}& + n(n-1)\lambda \wedge (\CL_{Y_1^\pi}d\lambda) \wedge (\CL_{Y_2^\pi}d\lambda) \wedge (d\lambda)^{n-2}.
\eeastar

The remaining section will be occupied by the proof of Proposition \ref{prop:Hessian-big}.
 
\begin{proof}[Proof of Proposition \ref{prop:Hessian-big}] 
Put $g=1+\log f$ and evaluate
$$
\int_M g h_1 \nabla \cdot Y_2^\pi \, d\mu_\lambda = \int_M g h_1 \CL_{Y_2^\pi}(d\mu_\lambda)
= - \int_M Y_2^\pi[ g h_1]\, d\mu_\lambda
$$
and similarly
$$
\int_M g h_2 \nabla \cdot Y_1^\pi \, d\mu_\lambda 
= - \int_M Y_1^\pi[g h_2]\, d\mu_\lambda.
$$
This computes the integral
$$
n \int_M g(h_1 \nabla \cdot Y_2^\pi + h_2 \nabla\cdot Y_1^\pi) \, d\mu_\lambda
=- n \int_M ( Y_2^\pi[g h_1] +  Y_1^\pi[g h_2])\, d\mu_\lambda.
$$
By adding the integral another term  $g (Y_2^\pi[h_1] +  Y_1^\pi[h_2])$,
this computation give rises to the third term of \eqref{eq:Hessian-big}.

Now we put $g = 1+ \log f$.
Most of the remaining section will be occupied by the evaluation of 
\be\label{eq:3rd-term}
\int_M n(n-1)g \lambda \wedge (\CL_{Y_1^\pi}d\lambda) \wedge (\CL_{Y_2^\pi}d\lambda) \wedge (d\lambda)^{n-2}.
\ee
We rewrite
\bea\label{eq:3rd-1}
&{}& n(n-1)g \lambda \wedge (\CL_{Y_1^\pi}d\lambda) \wedge (\CL_{Y_2^\pi}d\lambda) \wedge (d\lambda)^{n-2}
\nonumber\\
& = & n\lambda \wedge (g \CL_{Y_1^\pi}d\lambda) \wedge \CL_{Y_2^\pi}(d\lambda)^{n-1} \nonumber\\
& = & n \CL_{Y_2^\pi} \left(g \lambda \wedge (\CL_{Y_1^\pi}d\lambda) \wedge (d\lambda)^{n-1}\right)
\nonumber \\
&{}& - n \CL_{Y_2^\pi} \left(g \lambda \wedge (\CL_{Y_1^\pi}d\lambda)\right) \wedge (d\lambda)^{n-1}.
\eea
The first term of \eqref{eq:3rd-1} will vanish after integration.
For the second term thereof, we rewrite
\bea\label{eq:3rd-1-2}
&{}& - n \CL_{Y_2^\pi} \left(g \lambda \wedge (\CL_{Y_1^\pi}d\lambda)\right) \wedge (d\lambda)^{n-1}
\nonumber\\
& = & -n g \lambda \wedge \CL_{Y_2^\pi} (\CL_{Y_1^\pi}d\lambda)\wedge (d\lambda)^{n-1}
-  \CL_{Y_2^\pi}(g\lambda) \wedge \CL_{Y_1^\pi} (d\lambda)^n.
\eea
We rewrite the second term of \eqref{eq:3rd-1-2} into
$$
-  \CL_{Y_2^\pi}(g\lambda) \wedge \CL_{Y_1^\pi} (d\lambda)^n\\
=  - \CL_{Y_1^\pi}\left(\CL_{Y_2^\pi}(g\lambda) \wedge (d\lambda)^n\right)
 +  \left(\CL_{Y_1^\pi}\CL_{Y_2^\pi} (g\lambda) \right) \wedge (d\lambda)^n.
$$
On the other hand, we rewrite the first term of \eqref{eq:3rd-1-2} into
\beastar
&{}&
 -n (\CL_{Y_2^\pi} \CL_{Y_1^\pi}d\lambda)\wedge g\lambda \wedge(d\lambda)^{n-1}\\
 & = & -n d(\CL_{Y_2^\pi} \CL_{Y_1^\pi}\lambda)\wedge g\lambda \wedge(d\lambda)^{n-1}\\
 & = & -n d\left((\CL_{Y_2^\pi} \CL_{Y_1^\pi}\lambda)\wedge g\lambda \wedge(d\lambda)^{n-1}\right)
  -n \left((\CL_{Y_2^\pi} \CL_{Y_1^\pi}\lambda)\wedge d(g\lambda) \wedge(d\lambda)^{n-1}\right)\\
   & = & -n d\left((\CL_{Y_2^\pi} \CL_{Y_1^\pi}\lambda)\wedge g\lambda \wedge(d\lambda)^{n-1}\right)
  -n g((\CL_{Y_2^\pi} \CL_{Y_1^\pi}\lambda)\wedge (d\lambda)^n\\
  &{}&  -n (\CL_{Y_2^\pi} \CL_{Y_1^\pi}\lambda)\wedge dg \wedge \lambda \wedge (d\lambda)^{n-1}
\eeastar
Substituting the result into the first and the second terms of \eqref{eq:3rd-1}, after integration, we have derived
\bea\label{eq:3rd-3}
&{}&
\int_M n(n-1)g \lambda \wedge (\CL_{Y_1^\pi}d\lambda) \wedge (\CL_{Y_2^\pi}d\lambda) \wedge (d\lambda)^{n-2}
\nonumber\\
& = & \int_M \left(\CL_{Y_1^\pi}\CL_{Y_2^\pi}(g\lambda )\right) \wedge (d\lambda)^n 
 -n \int_M g  (\CL_{Y_2^\pi} \CL_{Y_1^\pi} \lambda) \wedge (d\lambda)^n \nonumber\\
&{}& - n\int_M \left((\CL_{Y_2^\pi} \CL_{Y_1^\pi}\lambda)\wedge dg\wedge \lambda \wedge
(d\lambda)^{n-1}\right)
\eea
Now we compute each of the three summands separately.

For the third term, we write $dg = X_g \intprod d\lambda + R_\lambda[g] \, \lambda$ and compute
\bea\label{eq:3rd-3-3}
&{}& 
-n (\CL_{Y_2^\pi} \CL_{Y_1^\pi}\lambda)\wedge dg\wedge \lambda \wedge
(d\lambda)^{n-1} \nonumber\\
& = & n  (\CL_{Y_2^\pi} \CL_{Y_1^\pi}\lambda)\wedge  \lambda \wedge (X_g^\pi \intprod d\lambda) \wedge (d\lambda)^{n-1}\nonumber\\
& = &  (\CL_{Y_2^\pi} \CL_{Y_1^\pi}\lambda) \wedge  \lambda \wedge (X_g^\pi \intprod  (d\lambda)^n)
\nonumber \\
& = & - \left( X_g^\pi \intprod (\CL_{Y_2^\pi} \CL_{Y_1^\pi}\lambda) \right) \, \mu_\lambda
+  (\CL_{Y_2^\pi} \CL_{Y_1^\pi}\lambda) \wedge (X_g^\pi \intprod  \lambda) \wedge  (d\lambda)^n
\nonumber \\
 & = & - \left( X_g^\pi \intprod (\CL_{Y_2^\pi} \CL_{Y_1^\pi}\lambda) \right) \, \mu_\lambda.
  \eea
Repeatedly using Cartan's formula and $Y_i^\pi \in \xi$, we further compute the first term
of \eqref{eq:3rd-3-3}
\beastar
X_g^\pi \intprod (\CL_{Y_2^\pi} \CL_{Y_1^\pi}\lambda) & = & (\CL_{Y_2^\pi} \CL_{Y_1^\pi}\lambda)(X_g^\pi)\\
& = &  \CL_{Y_2^\pi}( \CL_{Y_1^\pi}\lambda(X_g^\pi)) - (\CL_{Y_1^\pi} \lambda)([Y_2^\pi,X_g^\pi]) \\
& = &  \CL_{Y_2^\pi}(d\lambda(Y_1^\pi,X_g^\pi))- d\lambda(Y_1^\pi, [Y_2^\pi,X_g^\pi]) \\
& = &  -\CL_{Y_2^\pi}(\lambda([Y_1^\pi,X_g^\pi]))  - d\lambda(Y_1^\pi, [Y_2^\pi,X_g^\pi]) \\
& = & -d\lambda(Y_2^\pi, [Y_1^\pi,X_g^\pi]))- \lambda([Y_2^\pi, [Y_1^\pi,X_g^\pi])])
 - d\lambda(Y_1^\pi, [Y_2^\pi,X_g^\pi]).
\eeastar
Therefore we obtain
\bea\label{eq:3rd-3-4}
&{}& 
\int_M - n (\CL_{Y_2^\pi} \CL_{Y_1^\pi}\lambda)\wedge dg\wedge \lambda \wedge
(d\lambda)^{n-1} \nonumber\\
& = & \int_M  d\lambda(Y_2^\pi, [Y_1^\pi,X_g^\pi])) + \lambda([Y_2^\pi, [Y_1^\pi,X_g^\pi])])
 + d\lambda(Y_1^\pi, [Y_2^\pi,X_g^\pi])\, d\mu_\lambda.
\eea
 
For the evaluation of the second term
of \eqref{eq:3rd-3},  we rewrite
\beastar
(\CL_{Y_2^\pi}\CL_{Y_1^\pi} \lambda)(R_\lambda) & = & \CL_{Y_2^\pi}(\CL_{Y_1^\pi}\lambda(R_\lambda)) 
- (\CL_{Y_1^\pi}\lambda)([Y_2^\pi,R_\lambda]) \\
& = & \CL_{Y_2^\pi}(\CL_{Y_1^\pi}\lambda(R_\lambda)) + \lambda([Y_1^\pi, [Y_2^\pi,R_\lambda]])
=  \lambda([Y_1^\pi, [Y_2^\pi,R_\lambda]])
\eeastar
 since $\CL_{Y_1^\pi}\lambda(R_\lambda) = \CL_{Y_1^\pi}(\lambda(R_\lambda)) - \lambda([Y_1^\pi,R_\lambda]) = 0$ 
as $\lambda(R_\lambda) = 1$ and $[Y_1^\pi,R_\lambda] \in \xi$. Therefore we obtain
\be\label{eq:3rd-3-2}
-n \int_M g (\CL_{Y_2^\pi}\CL_{Y_1^\pi}\lambda)\wedge (d\lambda)^n 
 = - n \int_M  g \lambda([Y_1^\pi, [Y_2^\pi,R_\lambda]]) \, d\mu_\lambda.
\ee

For the first integral of \eqref{eq:3rd-3}, we rewrite its integrand into
\beastar
&{}&
 \left(\CL_{Y_1^\pi}\CL_{Y_2^\pi}(g\lambda) \right) \wedge (d\lambda)^n\\
& = & \CL_{Y_1^\pi}\left(g \CL_{Y_2^\pi}\lambda + \CL_{Y_2^\pi}(g)\lambda\right) \wedge (d\lambda)^n \\
& = & g(\CL_{Y_1^\pi}\CL_{Y_2^\pi}\lambda) \wedge (d\lambda)^n +( \CL_{Y_1^\pi}g) (\CL_{Y_2^\pi}\lambda) \wedge (d\lambda)^n \\
&{}& + \CL_{Y_1^\pi}(\CL_{Y_2^\pi}(g))\lambda \wedge (d\lambda)^n +(\CL_{Y_2^\pi}g) (\CL_{Y_1^\pi}\lambda) \wedge (d\lambda)^n \\
& = & g (\CL_{Y_1^\pi}\CL_{Y_2^\pi}\lambda) \wedge (d\lambda)^n+ \CL_{Y_1^\pi}(\CL_{Y_2^\pi}(g))\lambda \wedge (d\lambda)^n
\eeastar
where the last equality comes from Lemma \ref{lem:useful1}.
The first term, after integration, becomes
$$
\int_M  g \CL_{Y_1^\pi}\CL_{Y_2^\pi}\lambda \wedge (d\lambda)^n = \int_M  g \lambda([Y_2^\pi, [Y_1^\pi,R_\lambda]]) \, d\mu_\lambda
$$
 similarly as above with 1 and 2 swapped
in the evaluation of the third term. 
For the second term, we compute
\beastar
 \CL_{Y_1^\pi}(\CL_{Y_2^\pi}(g)) & =  & \CL_{Y_1^\pi}(d\lambda(X_g,Y_2^\pi))
= - \CL_{Y_1^\pi}(\lambda([X_g^\pi,Y_2^\pi]) \\
& = & - (\CL_{Y_1^\pi}\lambda) ([X_g^\pi,Y_2^\pi]) 
 -  \lambda([Y_1^\pi,[X_g^\pi, Y_2^\pi]]) \\
& = &
 - d\lambda(Y_1^\pi, [X_g^\pi,Y_2^\pi]) -  \lambda([Y_1^\pi,[X_g^\pi, Y_2^\pi]]).
\eeastar
Therefore
\bea\label{eq:3rd-3-1}
\int_M \left(\CL_{Y_1^\pi}\CL_{Y_2^\pi}(g\lambda )\right) \wedge (d\lambda)^n 
= \int_M  g \lambda([Y_2^\pi, [Y_1^\pi,R_\lambda]]) \, d\mu_\lambda\nonumber\\
 -\int_M  (d\lambda(Y_1^\pi, [X_g^\pi,Y_2^\pi]) + \lambda([Y_1^\pi,[X_g^\pi, Y_2^\pi]]))\, d\mu_\lambda.
\eea

Combining \eqref{eq:3rd-3-4}, \eqref{eq:3rd-3-2} and \eqref{eq:3rd-3-1}, we have derived 
\bea\label{eq:3rd-4}
&{}& 
\int_M n(n-1)g \lambda \wedge (\CL_{Y_1^\pi}d\lambda) \wedge (\CL_{Y_2^\pi}d\lambda) \wedge (d\lambda)^{n-2} \nonumber\\
& = &    \int_M (d\lambda(Y_2^\pi,[Y_1^\pi,X_g^\pi] + d\lambda(Y_1^\pi,[Y_2^\pi,X_g^\pi]) \, d\mu_\lambda \nonumber \\
&{}& +\int_M  (\lambda([Y_2^\pi,[Y_1^\pi,X_g^\pi]]) + \lambda([Y_1^\pi,[Y_2^\pi,X_g^\pi]]) \, d\mu_\lambda \nonumber\\
&{}& + \int_M-  n g \lambda([Y_1^\pi, [Y_2^\pi,R_\lambda]]) +  g \lambda([Y_2^\pi, [Y_1^\pi,R_\lambda]]) \, d\mu_\lambda
\nonumber \\
&{}& +\int_M d\lambda(Y_1^\pi,[Y_2^\pi,X_g^\pi]]\big)\, d\mu_\lambda.
\eea
\end{proof}


\section{Proof of Legendrian property of thermodynamic equilibrium}
\label{sec:Legendrian-property}

In this appendix, we give the proof of Legendrian property of the immersion
$\iota_{\CS_{\lambda_0}}: \Sigma_{\CS_{\lambda_0};\CF} \to J^1B_\CF$. This applies 
both to the case of the big and small phase spaces.

This is a standard fact in the finite dimensional case. (See \cite[Section 2.3]{duistermaat},
 for example, for the case of Symplectic/Lagrangian.)
Since we are dealing with the infinite dimensional situation, we provide a coordinate
independent covariant proof, for readers' convenience.

Denote by $\theta= \sum_{i=1}^N p_i dq_i$ the canonical Liouville one-form on the cotangent bundle $T^*B_\CF$,
and by $\pr: J^1 B_{\CF} \to T^*B_{\CF}$ and $\pi: T^* B_{\CF} \to B_{\CF}$ the natural projections.

\begin{prop}[Compare with Proposition 6.1 in \cite{lim-oh}] \label{prop:potential}
Let $$
\Lambda = dz - \pr^* \theta
$$
be the standard contact form on the one-jet bundle $J^1B_{\CF}$.
The contact thermodynamic equilibrium $R_{\CS_{\lambda_0};\CF}$ is a finite dimensional
(immersed) Legendrian submanifold of $J^1B_\CF$ of dimension $N$.
\end{prop}
\begin{proof} 
We consider the map
$$
\iota_{\CS_{\lambda_0}}  : \Sigma_{\CS_{\lambda_0};\CF}  \to J^1 B_\CF
$$
defined by
\be\label{eq:iota-defn}
\iota_{\CS_{\lambda_0}}  (\lambda) 
= \left(\CO_{\CF}(\lambda), D^h \CS_{\lambda_0} (\lambda), \CS_{\lambda_0}(\lambda)\right).
\ee
Here we use the identification
$$
(HT_{\lambda} \mathfrak C(M))^* \cong T^*_{\CO_{\CF}(\lambda)} B_\CF
$$
obtained by the Ehresman connection $T\mathfrak C^+(M) = HT\mathfrak C^+(M) \oplus VT\mathfrak C^+(M)$,
and the equality 
$$
D^h \CS_{\lambda_0} (\lambda) = d\CS_{\lambda_0} (\lambda)
$$
at $\lambda \in \Sigma_{\CS_{\lambda_0};\CF}$ 
on $\Sigma_{\CS_{\lambda_0};\CF}$ which proves that the value does not depend on the choice of connection. 
This shows that the map \eqref{eq:iota-defn} is well-defined.

Locally near a given $\lambda \in \Sigma_{\CS_{\lambda_0};\CF}$, we can express the 
image point $\iota_{\CS_{\lambda_0}}  (\lambda)$ in $J^1B_\CF \cong J^1\R^N$ 
in terms of the canonical coordinates  $(q_1,\cdots, q_N,p_1,\cdots, p_N,z)$: It is determined by
$$
q_i = \CO_{\CF_i}(\lambda), \quad i = 1, \cdots, N, \quad z = \CS_{\lambda_0}(\lambda)
$$
and
$(p_1, \cdots, p_N)$ is the Lagrange multiplier satisfying (\ref{eq:conormal-condition}).

Now, it suffices to show that $(\iota_{\CS_{\lambda_0}} )^*(\Lambda)=0$.
If we set $f = \CS_{\lambda_0}|_{\Sigma_{\CS_{\lambda_0};\CF} }$, then
we have
\be\label{eq:pi3}
(\iota_{\CS_{\lambda_0}} )^*(dz) = df.
\ee

Now we examine the map $\pr \circ \iota_{ \CS_{\lambda_0}} : \Sigma_{\CF ;\CS_{\lambda_0}}  \to T^*B_\CF$.
We have
$$
(\iota_{\CS_{\lambda_0}} )^*(\pr^* \theta) = (\pr \circ \iota_{\CS_{\lambda_0}} )^{*}\theta.
$$
By definition, we have $d^v\CS_{\lambda_0}(\lambda) = 0$ 
for all $\lambda \in \Sigma_{\CS_{\lambda_0};\CF} $. In particular, we have
$$
D^h\CS_{\lambda_0}(\lambda ) = d\CS_{\lambda_0}(\lambda).
$$
We now evaluate $\theta$ against $v \in T_\lambda (\Sigma_{\CS_{\lambda_0};\CF} )$ with $\beta \in T^*_q B_\CF$
\beastar
(\pr \circ \iota_{\CS_{\lambda_0}} )^*\theta(\lambda)(v) & = & 
\theta\left(\pr\left(\iota_{\CS_{\lambda_0}} (\lambda)\right)\left(d(\pr \circ \iota_{\CS_{\lambda_0}} )(v)\right)\right)\\
& = & \pr\left(\iota_{\CS_{\lambda_0}} (\lambda)\right) \left(d\pi \circ d(\pr \circ \iota_{\CS_{\lambda_0}} )(v)\right)\\
& = & \pr\left(\iota_{\CS_{\lambda_0}} (\lambda)\right) \left(d(\pi \circ \pr \circ \iota_{\CS_{\lambda_0}} )(v)\right)\\
& = & D^h\CS_{\lambda_0}(\lambda) \left( v \right).
\eeastar
Then we compute
\beastar
D^h\CS_{\lambda_0}(\lambda) (v)= d\CS_{\lambda_0}(\lambda) (v) = df(v)
\eeastar
for the function $f =\CS_{\lambda_0}|_{\Sigma_{\CS_{\lambda_0};\CF} }$. This proves
\be\label{eq:pi2}
(\iota_{\CS_{\lambda_0}} )^*\pr^*\theta = df
\ee
on $\Sigma_{\CF ; \CS_{\lambda_0}} $. Combining \eqref{eq:pi2} and \eqref{eq:pi3}, we have proved
$$
(\iota_{\CS_{\lambda_0}} )^*(\Lambda) = (\iota_{\CS_{\lambda_0}} )^*(dz) - (\iota_{\CS_{\lambda_0}} )^*(\pr^*\theta) = 0
$$
which finishes the proof of $(\iota_{\CS_{\lambda_0}})^*(\Lambda) = 0$.
\end{proof}

\begin{rem} As mentioned before, 
the preimage $\CO_{\CF}^{-1}(q)$ of the observation value $q =\CO_{\CF}(\lambda)$
contains multiple points some of which could have the same set of 
Lagrange multiplier $(p_1, \cdots, p_N)$ and the same entropy value
$\CS_{\lambda_0}(\lambda)$. In fact, this is the general reason why the Legendrian map
like \eqref{eq:iota-defn} can be an immersion, not an embedding in the generation of
Legendrian submanifolds in the case of Contact/Legndrian as well as in the case of Symplectic/Lagrangian 
in symplectic geometry.  However unlike the case of Symplectic/Lagrangian, thanks to the additional
condition for two points $\lambda, \, \lambda' \in \Sigma_{\CS_{\lambda_0;\CF}}$ 
$$
\CS_{\lambda_0;\CF}(\lambda) = \CS_{\lambda_0;\CF}(\lambda')
$$
to have the same image point under the Legendrian map $\iota_{\CS_{\lambda_0};\CF}$,
the occurrence of an immersed Legendrian submanifold is not a stable phenomenon
but a phenomenon of codimension 1 in the case of Contact/Legendrian 
while it is a stable one for the case of Symplectic/Lagrangian:
any immersed Legendrian submanifold can be made embedded by a $C^\infty$-small
perturbation. However they can appear in a generic one-parameter family of  Legendrian maps.
\end{rem}

\section{A sneak preview of \cite{do-oh:formalism}}
\label{sec:peek}

Here we would like to provide some outline on what kind of questions
the sequel \cite{do-oh:formalism} will systematically study via the thermodynamic formalism of
the contact dynamics  as a continuation of the present paper. A good summary for the standard
thermodynamic formalism of general dynamical systems is given in \cite{bruin} which we
refer non-expert readers thereto for definitions of various terms used in this appendix.

From now on, we consider a contact triad $(M, \lambda, J)$ and equip $M$ with
the associated triad metric $g$.  Denote by $d$ the associated distance function on $M$.

\subsection{Contact thermodynamic potential}

Note that by definition, we have 
$$\psi^*\mu_{\lambda}=\mu_{\psi^* \lambda}=e^{(n+1)g_{(\psi; \lambda)}}\mu_{\lambda}.$$
The following lemma is well-known.

\begin{lem}
For any $\psi, \phi \in \Cont _+(M, \xi)$, we have 
\be\label{eq:iteration}
g_{(\psi \phi; \lambda)}=g_{(\psi; \lambda)} \circ \phi + g_{(\phi; \lambda)},
\ee
where $\lambda\in \mathfrak C(M, \xi)$.
\end{lem}
An iterative application of this formula gives rise to
\be\label{eq:g-iteration}
g_{(\psi \phi^N; \lambda)}=g_{(\psi; \lambda)}\circ \phi^N +\sum_{i=0}^{N-1} g_{(\phi; \lambda)}\circ \phi^i.
\ee

An immediate corollary of of the estimate of the growth rate of the iteration.
\begin{cor} Let $\| \cdot \|_\infty$ the $L_\infty$-norm of real-valued function. Then
 all $N \geq 1$, we have 
$$
\frac1N \|g_{(\psi^N;\lambda)}\|_\infty  \leq  \|g_{(\psi;\lambda)}\|_\infty < \infty.
$$
\end{cor}
\begin{proof} This is an immediate consequence of \eqref{eq:g-iteration} with $\psi = \id$
$$
\frac1N g_{(\phi^N; \lambda)}= \frac 1N\sum_{i=0}^{N-1}g_{(\phi; \lambda)}\circ {\phi}^i.
$$
Therefore 
$$
\frac1N \|g_{(\phi^N; \lambda)} \|_\infty \leq \frac1N \sum_{i=0}^N \|g_{(\phi; \lambda)}\circ {\phi}^i\|_\infty
=\frac1N \sum_{i=0}^{N-1} \|g_{(\phi; \lambda)}\|_\infty = \|g_{(\phi; \lambda)}\|_\infty.
$$
\end{proof}

In addition, the following property of conformal exponent function is observed in our preprint
\cite{oh:contacton-Legendrian-bdy} and shows that the cohomology class of 
group cocycle of $\text{Cont}_+(M, \xi)$ does not depend on the choice of contact forms but depend 
only on the contact structure $\xi$ i.e. for any $\lambda, \lambda' \in \EC^+(M,\xi)$ and $\psi \in \text{Cont}_+(M, \xi)$ there exsits $h \in C^\infty (M, \R)$ such that $g_{(\psi; \lambda')}= h \circ \psi + g_{(\psi; \lambda)}$.
\begin{prop}[Remark 2.5  \cite{oh:contacton-Legendrian-bdy}]\label{prop:cohomologous}
 Let $(M,\xi)$ be a  coorinetable contact manifold. The map $\varphi_\lambda : \text{\rm Cont}_+(M, \xi) \to {\rm C}^\infty(M, \R),$ given by $\psi \mapsto g_{(\psi ; \lambda)}$, 
defines a group cocyle. And $\varphi_\lambda$ and $\varphi_{\lambda'}$ are cohomologous to each other for any $\lambda, \lambda' \in \EC^+(M, \xi)$.
\end{prop}
These motivate us to name the conformal exponent the \emph{thermodynamic potential}
of the pair $(\psi, \lambda)$ or of $\psi$ with respect to $\lambda$. 

Motivated by its iteration  property \eqref{eq:iteration} of conformal exponent $g_{(\psi;\lambda)}$ 
its relationship  \eqref{eq:gpsiH=intRlambda} with the dissipativity of contact Hamiltonian flow $\psi_H^t$, 
we now rename $g_{(\psi;\lambda)}$ the \emph{thermodynamic potential} of the pair $(\psi, \lambda)$ 
or of $\psi$  with respect to $\lambda$.

\subsection{Topological entropy and topological pressure}

As often the case in Hamiltonian dynamics, it is natural to consider the whole Hamiltonian flow
instead of considering an isolated diffeomorphism. In this regard, we consider a smooth function
$H$ and its contact Hamiltonian flow.
Then following integrated formula of $g_{(\psi_H^1;\lambda)}$  is well-known to the experts,
which provides a link between the dissipativity and the conformal exponent 
or the  thermodynamic potential $g_{(\psi; \lambda)}$.  
\be\label{eq:gpsiH=intRlambda}
g_{(\psi_H^t ; \lambda)}(x)=\int_0^t -R_{\lambda}[H](\psi_H^u(x))\,du.
\ee
The dissipative property of contact dynamics generated by
 a Hamiltonian $H$ is encoded by the Lie derivative of the volume form $\mu_\lambda$
 in  the direction of the contact Hamiltonian flow $\psi_H^t$. The following 
 formula
 \be
\CL_{X_H^\lambda}\mu_{\lambda}=-(n+1)R_{\lambda}[H]\, \mu_\lambda
\ee
provides the precise rate of change of the $\mu_\lambda$-measure.

We now link the contact thermodynamic potential function
$g_{(\psi;\lambda)}$ to  the well-known  concept of the \emph{ topological pressure 
with the \emph{temperature} $T = \frac1\beta$} as follows.
(See e.g., \cite{ruelle-book}, \cite{bowen,bowen-book}, \cite{bowen-ruelle} for further expounding
thereof.) One important notion entering in this general definition is the \emph{$N$-th partition function}.

\begin{defn}[The $N$-th partition function]
Let $\psi \in \text{Cont}_+(M, \xi)$. For each integer $N \geq 0$ and $\varepsilon>0$, 
a finite subset $E$ of $M$ is $(N, \varepsilon)$-separated if 
$$ 
x,\, y \in E \, , \, x \neq y \Longrightarrow d(\psi^k(x), \psi^k (y)) > \varepsilon \; \text{for some } 
0 \leq k \leq N.
$$
Then for each given $\beta \in \R$, we define
\begin{small}
\beastar
\CZ_{(\psi;\lambda); N} (\beta,\epsilon) &=&
\text{sup }\left\{\sum_{x\in E} \exp\left({-\beta (n+1)g_{(\psi;\lambda)}(x)}\right) \,\Big|\, E \text{ is } (N,\varepsilon)\text{-separated} \right\}, \label{eq:CZ}\\
\CZ_N(\beta, \epsilon) & =  & \text{sup }\left\{\sum_{x\in E} \exp\left(-\beta (n+1)S_N(g_{(\psi;\lambda)}(x) ) \right) \,
\Big|\, E \text{ is } (N,\varepsilon)\text{-separated} \right\} \label{eq:ZN}
\eeastar
\end{small}
where we have
\be\label{eq:SN}
S_N (g_{(\psi;\lambda)})  : = \sum_{k=1}^N {g_{(\psi;\lambda)}}\circ \psi^{k-1}.
\ee
\end{defn}

What is so special about the potential $g_{(\psi;\lambda)}$ against the general choice of 
a potential function $g: M \to \R$ is the identity 
$$
S_N(g_{(\psi;\lambda)} )= g_{(\psi^N;\lambda)}
$$
arising from \eqref{eq:g-iteration}, which in turn leads to 
\be\label{eq:ZN=ZpsiN}
\CZ_N = \CZ_{(\psi^N;\lambda); N}.
\ee
\emph{This identity is a special 
property unique of the potential $g_{(\psi;\lambda)}$, not shared by other general
potential function $g$ in the literature of the thermodynamic formalism of
general dynamical systems.}

Using this partition function, we introduce  the following topological pressure 
associated to the contact thermodynamic potential $g_{(\psi;\lambda)}$.
\begin{defn}[Contact topological pressure]
Let $(\psi, \lambda)$ be a contact pair. We define
\beastar
P_{(\psi;\lambda)}^{\text{\rm top}}(\beta,\varepsilon)
&=&\limsup_{N \to \infty} \frac1N \log \CZ_N(\beta ,\varepsilon)\\[5pt]
P_{(\psi;\lambda)}^{\text{\rm top}} (\beta)
&=& \lim_{\varepsilon \to 0} P_{(\psi;\lambda)}^{\text{\rm top}}(\beta,\varepsilon).
\eeastar   
We call $P_{(\psi;\lambda)}^{\text{\rm top}} (\beta)$ the \emph{contact topological pressure}
of $(\psi, \lambda)$ of temperature $1/\beta$.
\end{defn}

\subsection{Contact thermodynamic potential and variational principle}

In this section,  we define the measure theoretic version of the \textit{pressure} associated to 
$g_{(\psi;\lambda)}$ as a \emph{potential} the notion of which
 is standard in the  thermodynamic formalism of general dynamical system. 
(Again see (See e.g., \cite{ruelle-book}, \cite{bowen,bowen-book}, \cite{bowen-ruelle}
for detailed explanations in the context of general dynamical systems.)
  
The following theorem is one of the indication of the correctness of our choice of $g_{(\psi;\lambda)}$.

\begin{thm}[\cite{do-oh:formalism}] The contact equlibrium state does not depend on the choice of
contact forms but depends only on the contact strcuture $\xi$.
\end{thm}
This follows from  Proposition \ref{prop:cohomologous} and the definition of equilibrium state.

\begin{defn}[Variational pressure and equilibrium state] 
For each contact pair  $(\psi, \lambda)$, we denote by $\CM(\psi)$ the set of $\psi$-invariant
probability measures on $M$. Then we
define the $(\psi, \lambda)$-pressure function
\be\label{eq:pressure}
P_{(\psi;\lambda)} (\beta) := \sup_{\nu \in \CM(\psi)} 
\left\{h_\nu(\psi) - \beta \, (n+1) \int_M g_{(\psi;\lambda)} d\nu\right\},
\ee
where $h_\nu (\psi)$ is the measure theoretic entropy of an $\psi$-invariant probability measure $\nu$.\\
We call a $\psi$-invariant probability measure a \emph{contact equilibrium state}
(or \emph{contact equilibrium measure}), if it assumes the pressure \eqref{eq:pressure}.
\end{defn}
For each fixed measure $\nu$, the function
$$
\beta \mapsto h_{\nu}(\psi) - \beta \, (n+1) \int g_{(\psi;\lambda)}\, d\nu
$$
is a straight line with slope $-(n+1)\int g_{(\psi;\lambda)}\,d\nu$ and 
abscissa $h_{\nu}(\psi)$.
 Then the pressure function $P_{(\psi;\lambda)}(\beta)$ is the envelope of all
these straight lines. From this, it immediately follows that the function 

$$
\beta \mapsto P_{(\psi;\lambda)} (\beta)
$$
is a continuous convex function.

It follows from \cite[Section 3]{bowen-ruelle}  that
\be\label{eq:Pms=Ptop}
 P_{(\psi;\lambda)}(\beta) = P_{(\psi;\lambda)}^{\text{\rm top}}(\beta).
\ee
We note that as $\beta \to 0$ the potential energy $(n+1)\int_M g_{(\psi;\lambda)}\,d\nu$ plays no role, 
and we are just maximizing the measure theoretic entropy. As  $\beta \to \infty$, the potential energy becomes more and more important. 

Furthermore $P_{(\psi;\lambda)}(\beta)$ is always finite because
$$
\left|\int_M g_{(\psi; \lambda)}\,d\nu \right|\leq \text{sup}_{x \in M} |g_{(\psi; \lambda)}(x)| < \infty \; \text{for }\nu \in \CM(\psi)
$$ 
and
$$
\text{sup}_{\nu \in \CM(\psi)} \{ h_{\nu}(\psi)\} < \infty:
$$
This finiteness is a consequence of \cite[Corollary 13]{bowen}.

\begin{ques} Does there exist an equilibrium state for some positive temperature?
If not all for contact pair $(\psi, \lambda)$, what is a sufficient condition of $(\psi, \lambda)$
for the existence of equilibrium states?
\end{ques}

Now we can define the contact adaptation of  the  \textit{Gibbs measure}. 
 Equip $M$ with a metric $d$ (e.g., the distance induced by the triad metric of
 a contact triad $(M, \lambda,J)$).
 For convenience of exposition, we define, for each given integer $N \geq 0$, 
 $$
 d_{(\psi,N)}(x,y):=\text{sup}_{0 \leq k \leq N } \left\{d\left (\psi^k(x), \psi^k(y)\right )\right\},\;
 x,y \in M,
 $$ 
and a $d_{(\psi,N)}$-closed ball with radius $\varepsilon$ ($>0$)
$$
B_{d_{(\psi,N)}}(x, \varepsilon):=\{y \in M \mid d_{(\psi,N)}(x, y)\leq \varepsilon \}, \; x\in M.
$$
Note that $d_{(\psi,N)}$ also induces the topology of $M$. 
 
With these definitions, we are now ready to give the following definition.
\begin{defn}[Gibbs measure of $(\psi;\lambda)$] 
Let $(\psi, \lambda)$ be a contact pair on $M$. A $\psi$-invariant Borel probability measure $\mu$ on $M$ is called a 
\textit{Gibbs measure for $\psi$ at temperature $1/\beta$} if  there exists constants $\varepsilon >0$, $C> 0$ 
and $P \in \R$ such that for  $d_{(\psi,N)}$-closed balls  $B_{(N,\varepsilon)}$ of $M$
with radius $\varepsilon$,  and all $x \in B_{(N,\varepsilon)}$, we have
\be\label{eq:defining-Gibbs}
\frac1C \leq \frac{\mu(B_{(N, \varepsilon)})}{\exp\{- \beta\,(n+1) g_{(\psi^N;\lambda)}(x)-N P\}} 
 \leq C.
\ee
for any $N\geq 0.$
\end{defn}

\def\cprime{$'$}
\providecommand{\bysame}{\leavevmode\hbox to3em{\hrulefill}\thinspace}
\providecommand{\MR}{\relax\ifhmode\unskip\space\fi MR }
\providecommand{\MRhref}[2]{%
  \href{http://www.ams.org/mathscinet-getitem?mr=#1}{#2}
}
\providecommand{\href}[2]{#2}

\end{document}